\documentclass[12pt,A4]{article}
\usepackage{amsfonts}
\usepackage{amsmath}
\usepackage{amsthm}
\usepackage{graphicx}

\newtheorem{theorem}{Theorem}[section]
\newtheorem{lemma}[theorem]{Lemma}

\numberwithin{equation}{section}
\newcommand{\ds}{\displaystyle}

\title{Zero distribution of orthogonal polynomials on a $q$-lattice\thanks{%
Work supported by EOS project 30889451 and FWO project G.0C9819N.}
}

\author{Walter Van Assche\footnote{Department of Mathematics, KU Leuven, Celestijnenlaan 200B box 2400, BE-3001 Leuven, Belgium.}
        \and Quinten Van Baelen}

\begin{document}

\maketitle

\begin{abstract}
We give the asymptotic behavior of the zeros of orthogonal polynomials, after appropriate scaling, for which the orthogonality measure is supported on the $q$-lattice $\{q^k, k=0,1,2,3,\ldots\}$, where $0 < q < 1$. The asymptotic distribution of the zeros is given by the radial part of the equilibrium measure of an extremal problem in logarithmic potential theory for circular symmetric measures with a constraint imposed by the $q$-lattice.
\end{abstract}

\section{Introduction}

Let $\{ p_n,\ n=0,1,2,3,\ldots\}$ be orthogonal polynomials satisfying a discrete orthogonality relation on the $q$-lattice
$\{ q^k, k=0,1,2,\ldots\}$:
\[   \sum_{k=0}^\infty p_n(q^k) p_m(q^k) q^k w(q^k) = \delta_{n,m},  \]
where $0 < q < 1$ and $w$ is a weight function on the $q$-lattice. In terms of the Jackson $q$-integral 
\[   \int_0^1 f(x)\, d_qx =  (1-q) \sum_{k=0}^\infty  f(q^k) q^k, \]
the orthogonality is
\[    \int_0^1  p_n(x)p_m(x) w(x) \, d_qx = (1-q) \delta_{n,m}.  \]
Typical examples of families of orthogonal polynomials on a $q$-lattice are the little $q$-Jacobi polynomials \cite[\S 14.12]{KLS},
the little $q$-Laguerre polynomials \cite[\S 14.20]{KLS},
and the $q$-Bessel polynomials \cite[\S 14.22]{KLS}, which we will consider in detail in Section \ref{sec:ex}.
The NIST handbook \cite[\S 18.27]{NIST} contains a section on the $q$-Hahn class but we recommend the book \cite{KLS} for a more
detailed collection of these discrete orthogonal polynomials on a $q$-lattice. 

In this paper we are interested in the asymptotic distribution of the zeros of such orthogonal polynomials.
It is well known that the zeros $x_{1,n} < x_{2,n} < \cdots < x_{n,n}$ are all in $[0,1]$, which is the smallest interval containing  
$\{q^k, k=0,1,2,\ldots\}$, and for discrete orthogonal polynomials there can be at most one zero between two consecutive points of the
discrete set, i.e., at most one zero of $p_n$ between $q^{k+1}$ and $q^k$ for every $k$ and every $n$, see, e.g., \cite[Thm. 6.1.1]{Szego}
or \cite[Thm. 2.4]{Freud}.
This means that if $q^{k+1} < a \leq q^k$, then there are at most $k+1 = \lceil \log a/\log q \rceil$ zeros of each $p_n$
in $[a,1]$. Hence most zeros are in $[0,a]$ for every $a >0$ and the zeros $\{x_{j,n}, 1 \leq j \leq n\}$ of $p_n$ accumulate
at $0$ in the sense that
\[  \lim_{n \to \infty} \frac{1}{n} \sum_{k=1}^n f(x_{k,n}) = f(0), \]
for every continuous function $f$ on $[0,1]$. 
This asymptotic distribution was also found in \cite{AlBuDeh} using the moments of the asymptotic zero distribution and the asymptotic behavior of the
coefficients in the three-term recurrence relation.
In order to find a more interesting distribution of the zeros, i.e., a limit distribution which
is not degenerate at $0$, we need some scaling of the zeros. Since the zeros $x_{j,n}$ accumulate at $0$ at least as fast as the points in the
lattice $\{q^n, n = 0,1,2,\ldots\}$, a proper scaling is to investigate the distribution of $\{x_{j,n}^{1/n}, 1 \leq j \leq n \}$ as $n \to \infty$.
This moves the zeros away from $0$ towards the endpoint $1$. We therefore will investigate the zeros of $p_n(x^n)$. This scaling and the
asymptotic distribution of the zeros of $p_n(x^n)$, where the polynomials are orthogonal on $[0,\infty)$ with a slowly decaying weight (such as the Stieltjes-Wigert polynomials \cite[\S 14.27]{KLS}), was first suggested and worked out by Kuijlaars in \cite{Kuijl}. Note that $x_{1,n}^{1/n},\ldots,
x_{n,n}^{1/n}$ are indeed zeros of $p_n(x^n)$, but this polynomial of degree $n^2$ has many more complex zeros
\[  x_{1,n}^{1/n} \omega_n^k, x_{2,n}^{1/n} \omega_n^k, \ldots, x_{n,n}^{1/n} \omega_n^k, \qquad 0 \leq k \leq n-1, \]
where $\omega_n = e^{2\pi i/n}$ is the primitive $n$th root of unity. 
\begin{figure}[ht]
\centering
\includegraphics[width=3in]{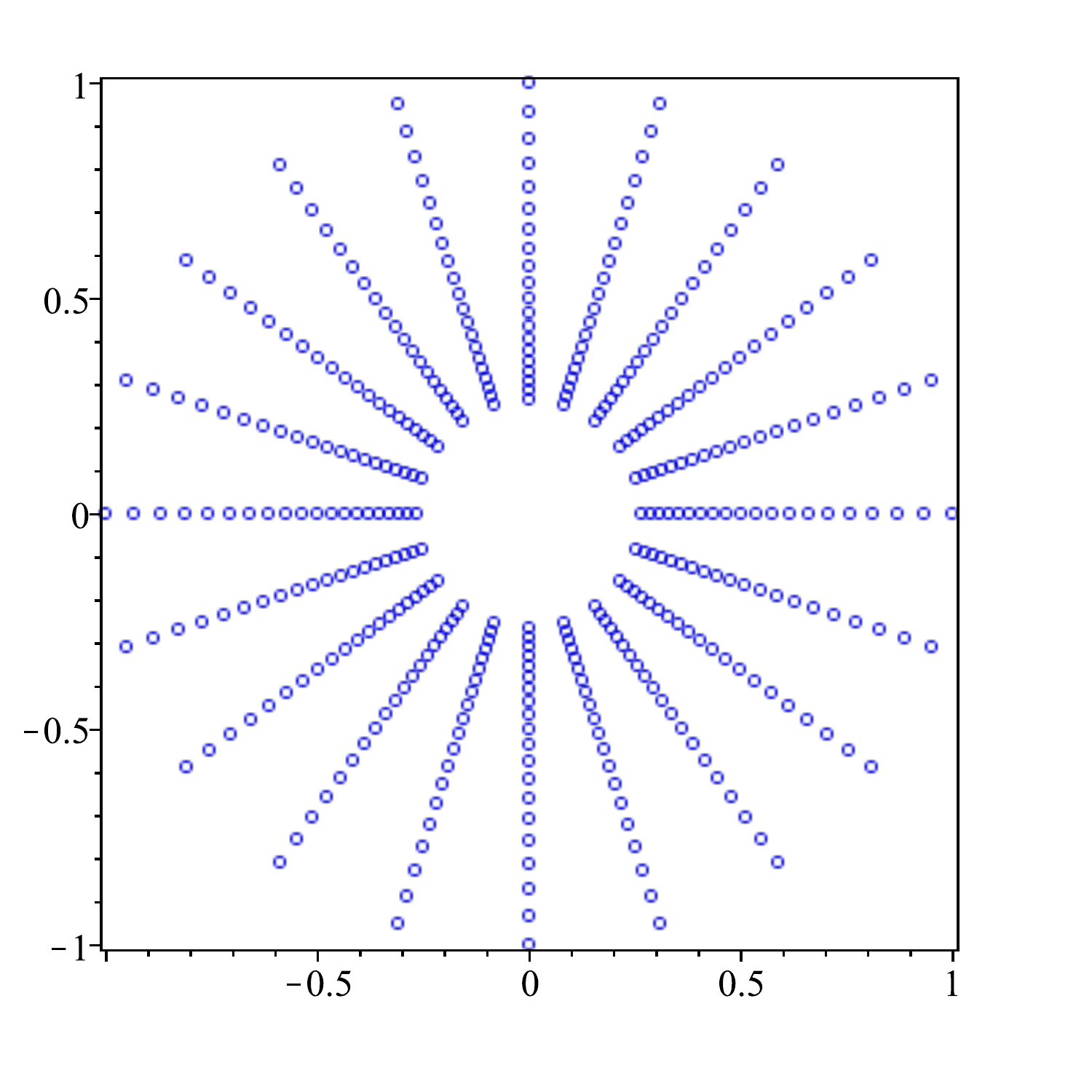}
\caption{The 400 zeros of $p_{20}(x^{20})$ for the little $q$-Jacobi polynomial with $q=1/4$, $a=b=1/2$.}
\label{fig:zeros}
\end{figure}

So our investigation of the asymptotic distribution of the zeros
of $p_n(x^n)$ will give a limiting measure in the complex plane with circular symmetry. Indeed, if we define the zero counting measure for
$p_n(x^n)$ by
\[   \mu_n = \frac{1}{n^2} \sum_{j=1}^n \sum_{k=0}^{n-1} \delta_{x_{j,n}^{1/n} \omega_n^k}, \]
where $\delta_a$ is the unit measure with mass 1 at the point $a$, then this gives a sequence of probability measures on the
closed unit disk $\overline{\mathbb{D}}= \{ z \in \mathbb{C}:\ |z| \leq 1 \}$ and we will show that it converges weakly to
a probability measure $\mu$ on  $\overline{\mathbb{D}}$ in the sense that for every continuous function $f$ on $\overline{\mathbb{D}}$
\[   \lim_{n \to \infty} \int f(x)\, d\mu_n(x) = \int f(x) \, d\mu(x).  \]
The circular symmetry gives
\[     \int f(x)\, d\mu(x) = \frac{1}{2\pi} \int_0^1 \int_0^{2\pi} f(re^{i\theta})\, d\theta d\nu(r),  \]
and the measure $\nu$ gives the radial part of the measure $\mu$ and is related to the asymptotic distribution of the positive real zeros
of $p_n(x^n)$: setting
\begin{equation}  \label{nu_n}
  \nu_n = \frac{1}{n} \sum_{j=1}^n \delta_{x_{j,n}^{1/n}}, 
\end{equation}
one has
\[   \lim_{n \to \infty} \int_0^1 g(x) \, d\nu_n(x) = \int_0^1 g(x)\, d\nu(x), \]
for every continuous function $g$ on $[0,1]$.

We will show that the limiting measure $\mu$ on the unit disk is given by the solution of an extremal problem in logarithmic potential
theory for circular symmetric measures on the unit disk satisfying a constraint. 
Such extremal problems with a constraint were first used by Rakhmanov to find the asymptotic distribution of the zeros
of discrete orthogonal polynomials on a uniform lattice (the discrete Chebyshev polynomials) in \cite{Rakh}. Later, Dragnev and Saff
extended these results to cover other discrete orthogonal polynomials \cite{DragSaff1}, in particular the Krawtchouk polynomials
\cite{DragSaff2}. Kuijlaars and Van Assche \cite{KuijlWVA} further extended this for discrete orthogonal polynomials on an unbounded set.
A nice survey of zeros of discrete orthogonal polynomials and equilibrium measures in logarithmic potential theory with constraints, 
can be found in \cite{KuijlRakh}.

\section{Main results}  \label{sec:main}
We will be using the norms
\[    \|f\|_{E} = \|f\|_{\infty,E} = \sup_{x \in E} |f(x)|, \qquad   \|f\|_{2,q} = \left( (1-q) \sum_{k=0}^\infty q^k |f(q^k)|^2 \right)^{1/2}.  \]
In order to prove our main results, we will use the following result \cite[Thm.~7.4]{KuijlWVA}
\begin{theorem} \label{thm:KVA}
Let $Q(x) = - \log w(x)$ be an admissible field $[0,\infty)$ and $\sigma$ an admissible constraint. Let $(E_n)_{n \in \mathbb{N}}$ be a sequence of subsets of $[0,\infty)$
satisfying the following conditions:
\begin{enumerate}
  \item For each $n$ and each compact set $A$ in $[0,\infty)$ the quantity $\sigma_n(A) = \frac1{n} \# (A \cap E_n)$ is finite.
  \item For every continuous function $f$ with compact support on $[0,\infty)$ one has
  \[   \lim_{n \to \infty} \int f(x)\, d\sigma_n(x) = \int f(x)\, d\sigma(x). \]  
   \item The set of points that are not normal for the sequence $(E_n)_{n \in \mathbb{N}}$ has logarithmic capacity zero.
\end{enumerate} 
Then the following hold:
\begin{itemize}
  \item If $T_n$ is a monic polynomial of degree $n$ minimizing $\|w^nP_n\|_{E_n}$ among all monic polynomials $P_n$ of degree $n$, then
  \[    \lim_{n \to \infty} \| w^nT_n\|_{E_n}^{1/n} = \exp(-w_Q^\sigma).  \]
   \item For every sequence of monic polynomials $(P_n)_{n \in \mathbb{N}}$ with $\deg P_n = n$, such that
   \[   \lim_{n \to \infty}  \|w^n P_n\|_{E_n}^{1/n} = \exp(-w_Q^\sigma), \]
   one has
   \[   \lim_{n \to \infty} \frac1{n} \sum_{k=1}^n \delta_{x_{k,n}} = \mu_Q^\sigma, \]
   where $\{x_{k,n}, 1 \leq k \leq n\}$ are the zeros of $P_n$.
\end{itemize}
\end{theorem}

This theorem is formulated in \cite{KuijlWVA} for functions on $[0,\infty)$, but since our $q$-lattice is in $[0,1]$, we will only be using the theorem for functions on $[0,1]$.
In that case, a field $Q$ is admissible if $Q$ is lower semi-continuous on $[0,1]$ and if the set $E_0=\{ x \in [0,1] : Q(x)< \infty\}$ has positive capacity. There is no need for
a growth condition as $x \to \infty$. A constraint $\sigma$ is admissible on $[0,1]$ if the support of $\sigma$ is $[0,1]$, $\sigma(E_0) >1$, and $\sigma$
has finite logarithmic energy over compact sets of $(0,1)$, i.e., for every compact set $K \subset (0,1)$ one has
\[  \int_K\int_K \log \frac{1}{|x-y|} \, d\sigma(x) d\sigma(y) < \infty.  \]
Finally, a point $x \in [0,1]$ is called a normal point for the sequence $(E_n)_{n \in \mathbb{N}}$ if there exists a constant $\rho_x >0$, a neighborhood $V_x$ of $x$ and
a positive integer $N_x$, such that for all $n \geq N_x$ and for every pair of different points $x_1,x_2 \in E_n \cap V_x$ one has
\[    |x_1-x_2| \geq \frac{\rho_x}{n}.  \]
In Section \ref{sec:3} we will show that we need the constraint $\sigma$ on $[0,1]$ with density
\[    \sigma'(t) = \frac{-1}{\log q} \, \frac{1}{t}, \]
which satisfies the conditions for an admissible constraint on $(0,1)$. Observe that $\sigma([0,1]) = \infty$ but $\sigma([a,1])$ is finite
whenever $0 < a < 1$. 

Our first main result is the asymptotic behavior of the leading coefficient $\gamma_n$ and the
asymptotic behavior of the $n$-th roots $x_{1,n}^{1/n} < x_{2,n}^{1/n} < \cdots < x_{n,n}^{1/n}$ of the zeros of the orthonormal polynomials $p_n$
when the weight function does not generate an external field.

\begin{theorem}     \label{thm:2.1}
Let $p_n(x)=\gamma_n x^n+\cdots$ be the orthonormal polynomials on the exponential lattice $\{q^n, n=0,1,2,3,\ldots\}$ with weights
$w_k=w(q^k) >0$ such that
\begin{equation}   \label{wzero}
   \lim_{n \to \infty} \frac{1}{n^2} \log w(x^n) = 0,  
\end{equation}
uniformly on every closed interval of $(0,1]$.
Then 
\begin{equation}  \label{normq}
  \lim_{n \to \infty} \gamma_n^{1/n^2} = q^{-1/2}, 
\end{equation}
and the distribution of the roots $x_{1,n}^{1/n} < \cdots < x_{n,n}^{1/n}$ is given by the measure $\nu^\sigma$ with density
\[   (\nu^\sigma)'(t) = \frac{-1}{\log q} \frac{1}{t}, \qquad t \in [q,1]. \]
\end{theorem}

We will prove this theorem in Section \ref{sec:4}

Our second main result is the case when the weight function generates an external field. In this case the asymptotic behavior
of the leading coefficient $\gamma_n$ and the asymptotic distribution of the scaled zeros $x_{1,n}^{1/n} < x_{2,n}^{1/n} < \cdots < x_{n,n}^{1/n}$
depends on the equilibrium measure $\mu_Q^\sigma$ on $\overline{\mathbb{D}}$ with constraint $\sigma$ and external field $Q$,
which has radial part $\nu_Q^\sigma$ and equilibrium constant $w_Q^\sigma$. These notions will be explained in Section \ref{sec:3}.

\begin{theorem}   \label{thm:2.2}
Let $p_n(x)=\gamma_n x^n+\cdots$ be the orthonormal polynomials on the exponential lattice $\{q^n, n=0,1,2,3,\ldots\}$ with weights
$w_k=w(q^k) >0$ such that
\begin{equation}   \label{wQ}
   \lim_{n \to \infty} \frac{-1}{n^2} \log w(x^n) = 2Q(x),   
\end{equation}
uniformly on every closed interval of $(0,1]$, 
with $Q$ an admissible external field on $[0,1]$ which is decreasing on $(0,\delta]$ for some $\delta > 0$.
Then
\[  \lim_{n \to \infty} \gamma_n^{1/n^2} = e^{w_Q^{\sigma}} , \]
and the distribution of the roots $x_{1,n}^{1/n} < \cdots < x_{n,n}^{1/n}$ is given by the density $\nu_Q^\sigma$.
\end{theorem}

The proof of this result will be given in Section \ref{sec:5}.

\section{Equilibrium problem with constraints}   \label{sec:3}
The polynomials $p_n(x^n)$ are orthogonal on the set
\[   E_n = \{ q^{k/n}, k=0,1,2,3,\ldots \} . \]
The distribution of the points in this set $E_n$ is given by the measure $\sigma_n$ for which
\[  \sigma_n([a,b]) = \frac{ \# \{k : a \leq q^{k/n} \leq b \}}{n}, \]
where $0 < a < b \leq 1$ and this is equal to
\[   \sigma_n([a,b]) = \frac{\# \{k : n \log a \leq k \log q \leq n \log b \}}{n} \]
and clearly one has
\[   \lim_{n \to \infty} \sigma_n([a,b]) = \frac{\log b - \log a}{-\log q} = \frac{-1}{\log q} \int_a^b \frac{dt}{t}.  \]
The positive real zeros of $q_n(x) = p_n(x^n)$ are separated by the points in $E_n$, which means that the measure $\nu_n$ given in
\eqref{nu_n} is bounded from above by the measure $\sigma_n$:
\[   \nu_n([a,b]) \leq \sigma_n([a,b]),  \]
and for $n \to \infty$ this gives a constraint on the limiting measures
\[   \nu \leq \sigma, \qquad  \sigma'(t) = \frac{-1}{\log q} \frac{1}{t}, \qquad t \in (0,1].  \]

\subsection{Constrained equilibrium problem}  \label{sec:3.1}
Following Rakhmanov \cite{Rakh}, Dragnev and Saff \cite{DragSaff1,DragSaff2}, Kuijlaars and Rakhmanov \cite{KuijlRakh}, we look at the
following constrained equilibrium problem. 
Let \[   I(\mu) = \iint \log \frac{1}{|x-y|} \, d\mu(x)d\mu(y), \]
be the logarithmic energy of a probability measure $\mu$ and
\[   U(z;\mu) = \int \log \frac{1}{|z-x|}\, d\mu(x) \]
its logarithmic potential. We will be using circular symmetric measures $\mu$ on the closed unit disk $\overline{\mathbb{D}}$ for
which
\[   \int_{\overline{\mathbb{D}}} f(z)\, d\mu = \frac{1}{2\pi} \int_0^1 \int_0^{2\pi} f(re^{i\theta})\, d\theta d\nu(r), \]
where the measure $\nu$ on $[0,1]$ is the radial part of the measure $\mu$, and in such a case we use the notation $\mu = \hat{\sigma}$.

The equilibrium problem is to find the infimum of $I(\mu)$ over all circular symmetric probability measures $\mu$ on
the closed unit disk $\overline{\mathbb{D}}$ with a radial component $\nu \leq \sigma$.  
The infimum exists, it is unique, and we denote it by $\mu^\sigma$. It is characterized by the variational inequalities
\begin{eqnarray}
   U(z;\mu^\sigma) & \geq & w^\sigma, \qquad z \in \textup{supp}(\hat{\sigma} - \mu^\sigma),  \label{var1} \\
   U(z;\mu^\sigma) & \leq & w^\sigma, \qquad z \in \textup{supp}(\mu^\sigma),   \label{var2}
\end{eqnarray}
where $w^\sigma$ is a constant. This equilibrium measure has the extremal property
\[  \min_{x \in \textup{supp}(\hat{\sigma}-\mu^\sigma)} U(x;\mu^\sigma) = \max_{\mu=\hat{\nu} \textup{ with } \nu \leq \sigma}
    \min_{x \in \textup{supp}(\hat{\sigma}-\mu} U(x;\mu).  \]

The solution of the unconstrained equilibrium problem on the unit disk $\overline{\mathbb{D}}$ is the equilibrium measure on
the closed unit disk, which is known to be the Lebesgue measure on the unit circle $\mathbb{T}$ (the boundary of the closed unit disk).
This is a circular symmetric measure with radial part the Dirac measure $\delta_1$ at $1$. Hence this is also the solution of the unconstrained
equilibrium problem for circular symmetric measures. The radial part, however, violates the constraint $\nu \leq \sigma$ (and quite
seriously). We therefore need to distribute the Dirac measure $\delta_1$ over $[0,1]$ so that it satisfies the constraint $\nu \leq \sigma$
and is a probability measure. This gives the circular symmetric measure $\mu^\sigma = \hat{\nu}^\sigma$ with radial part
\[   (\nu^\sigma)'(t) = \frac{-1}{\log q} \frac{1}{t}, \qquad t \in [q,1].  \]

\begin{figure}[h]
\centering
\includegraphics[width=3in]{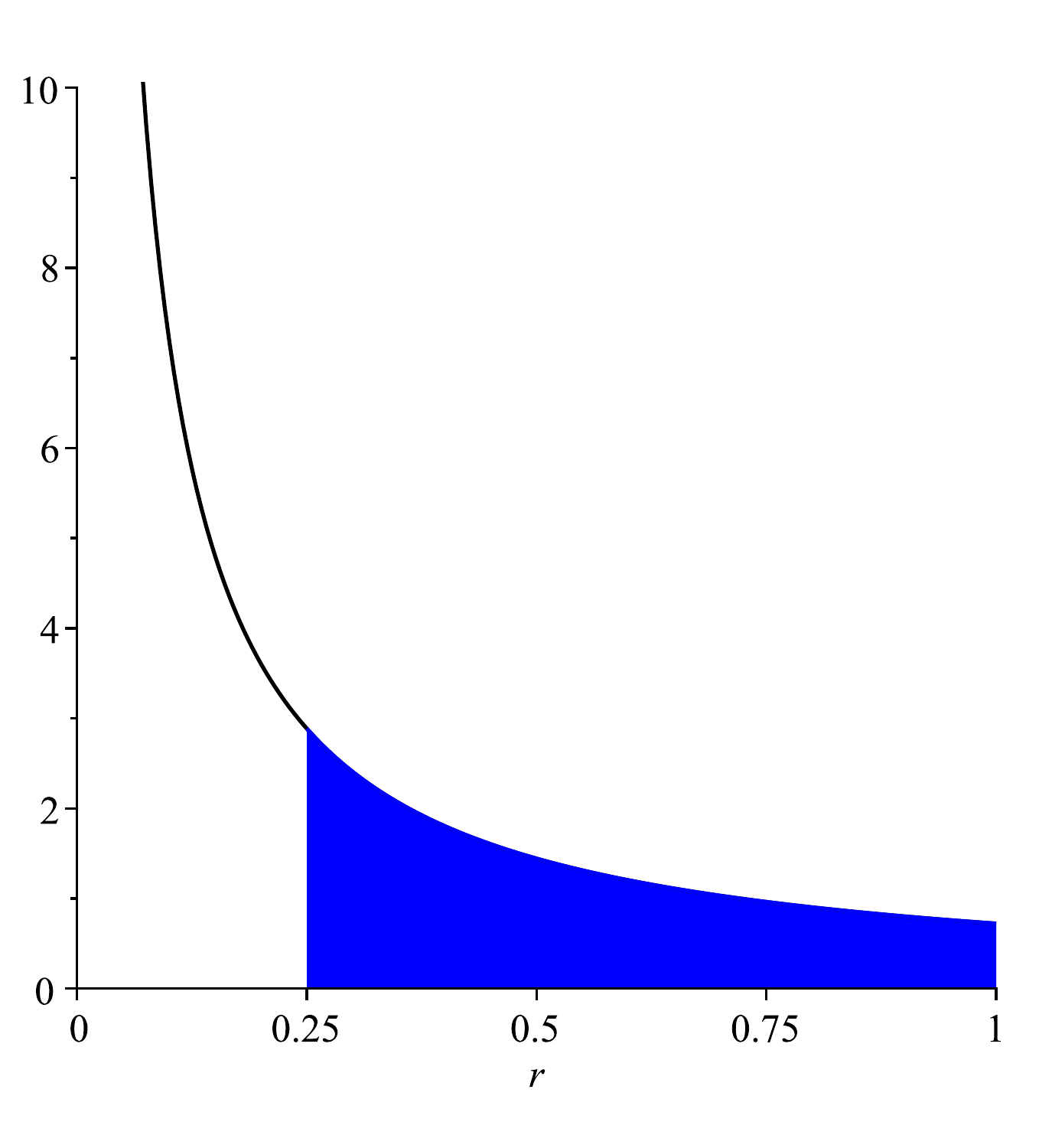}
\caption{The density of the radial part of the constrained equilibrium measure for $q=1/4$. The black curve is the constraint.}
\label{fig:2}
\end{figure}
 
Observe that the support of $\mu^\sigma$ is the annulus $\{q \leq |z| \leq 1 \}$ and that the support of $\hat{\sigma}-\mu^\sigma$ is the disk
$\{ |z| \leq q \}$. The radial part $\nu^\sigma$ coincides with the constraint $\sigma$ on the interval $[q,1]$ and has no mass on $[0,q]$.
Furthermore for $|z| \leq q$
\begin{eqnarray*}
    U(z;\mu^\sigma) &=& \frac{-1}{\log q}\, \frac{1}{2\pi} \int_0^{2\pi} \int_q^1 \frac{1}{r} \log \frac{1}{|z-re^{i\theta}|} \, dr d\theta \\
    &=& \frac{1}{\log q} \int_q^1 \frac{\log r}{r} \, dr \\
    &=& - \frac{\log q}{2}.
\end{eqnarray*}
For $|z|\geq 1$ one has
\[  U(z;\mu^{\sigma}) = - \log |z|,  \]
and for $q \leq |z| \leq 1$ one has
\begin{eqnarray*}
   U(z;\mu^\sigma) &=& \frac{1}{\log q} \left( \int_q^{|z|} \frac{\log |z|}{r}\, dr + \int_{|z|}^1 \frac{\log r}{r}\, dr \right) \\
                   &=& \frac{(\log |z|)^2}{2\log q} - \log |z|.
\end{eqnarray*} 
Hence the constant $w^\sigma$ in \eqref{var1}--\eqref{var2} is equal to $-\frac{1}{2} \log q$.

\begin{figure}[h]
\centering
\includegraphics[width=3in]{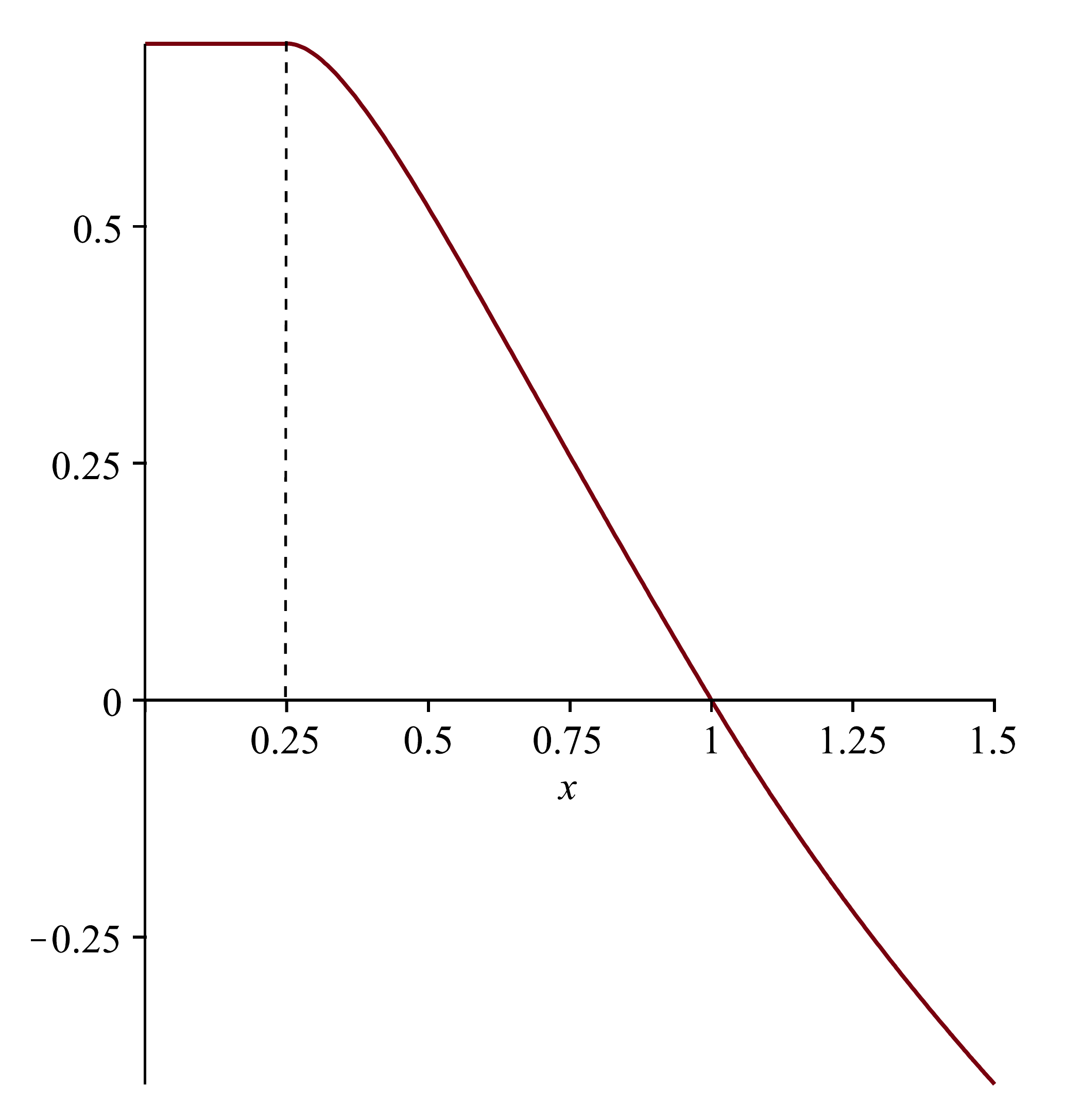}
\caption{The logarithmic potential $U(z;\mu^\sigma)$ for $q=1/4$ on $\mathbb{R}^+$.}
\label{fig:3}
\end{figure}

\subsection{Adding an external field}  \label{sec:3.2}

If we add an external field $Q$, then the equilibrium problem is to find the infimum of
\[   I_Q(\mu) = \iint \log \frac{1}{|x-y|} \, d\mu(x)d\mu(y) + 2 \int Q(|x|)\, d\mu(x)  \]
over all circular symmetric probability measures $\mu$ on the closed unit disk $\overline{\mathbb{D}}$ with
radial component $\nu \leq \sigma$. If $Q$ is admissible then the infimum exists, it is unique, and we denote it by $\mu_Q^\sigma$. It is characterized by the variational inequalities
\begin{eqnarray}
   U(z;\mu_Q^\sigma) + Q(|z|) & \geq & w_Q^\sigma, \qquad z \in \textup{supp}(\hat{\sigma} - \mu_Q^\sigma),  \label{Qvar1} \\
   U(z;\mu_Q^\sigma) + Q(|z|) & \leq & w_Q^\sigma, \qquad z \in \textup{supp}(\mu_Q^\sigma),   \label{Qvar2}
\end{eqnarray}
where $w_Q^\sigma$ is a constant.

The constrained equilibrium measure with external field $Q$ for circular symmetric measures can be computed explicitly in two cases:

\begin{enumerate}
  \item If $Q$ is a decreasing function on $(0,1]$. \\
   The equilibrium measure $\mu^\sigma$ without external field satisfies
   \begin{align*}
    U(z;\mu^\sigma) &= w^\sigma, \qquad  0 \leq |z| \leq q, \\
    U(z;\mu^\sigma) &\leq w^\sigma, \qquad  q \leq |z| \leq 1.
    \end{align*}
   Now use $Q(|z|) \geq Q(q)$ for $|z| \leq q$ and  $Q(|z|) \leq Q(q)$ for $q \leq |z|$ to find 
   \begin{align*}
    U(z;\mu^\sigma) + Q(|z|)  &\geq w^\sigma + Q(q), \qquad 0 \leq |z| \leq q, \\
    U(z;\mu^\sigma) + Q(|z|) &\leq w^\sigma + Q(q) , \qquad q \leq |z| \leq 1.
    \end{align*}
    These are the variational conditions \eqref{Qvar1}--\eqref{Qvar2}, hence the equilibrium measure is the one without external field ($\mu_Q^\sigma = \mu^\sigma$)
    but the Lagrange multiplier changes to  $w_Q^\sigma = w^\sigma + Q(q)$.
  \item If $Q$ is convex or $xQ'(x)$ is an increasing function on $(0,1]$. \\ 
   For the unconstrained equilibrium measure one has the following result \cite[Thm. 6.1 in Ch. IV]{SaffTotik} for an external field on $(0,\infty)$:
   \begin{theorem} \label{thm:ST}
   Let $Q$ be a circular symmetric field which is differentiable. 
   Suppose $xQ'(x)$ is increasing on $(0,\infty)$ or $Q$ is convex on $(0,\infty)$.
   Let $r_0$ be the smallest number for which $Q'(r) >0$ for all $r >r_0$ and $R_0$ the smallest solution of $xQ'(x)= 1$.
   Then $\mu_Q$ is supported on the annulus $r_0 \leq |z| \leq R_0$ and its radial part has density
   \[   (\nu_Q)'(t) = \bigl(tQ'(t)\bigr)', \qquad t \in [r_0,R_0].  \]
   \end{theorem} 
   First we adapt this result to an external field on $(0,1]$. 
\begin{itemize}
  \item If $R_0  \leq 1$ then the result of Theorem \ref{thm:ST} still holds: $\mu_Q$ is supported on the annulus $r_0 \leq |z| \leq R_0$ and its radial part has density
   \[   (\nu_Q)'(t) = \bigl(tQ'(t)\bigr)', \qquad t \in [r_0,R_0].  \]
  \item If $r_0 < 1$ and $xQ'(x)=1$ has no solution on $(0,1]$ (which corresponds to $R_0 > 1$), then $\mu_Q$ is supported on the annulus $r_0 \leq |z| \leq 1$ and
   the radial part has density
   \[  (\nu_Q)'(t) = \bigl(tQ'(t)\bigr)', \qquad t \in [r_0,1],  \]
   together with a discrete part $c\delta_1$ at $1$ with mass
   \[   c = 1 - \int_{r_0}^1 \bigl( tQ'(t) \bigr)' \, dt = 1 - Q'(1). \]
   The missing proportion $c$ on $(1,R_0]$ of the measure in Theorem \ref{thm:ST} is replaced by a dirac measure $c \delta_1$ at $1$.
   \item If $r_0  \geq 1$ then $\mu_Q$ is the Lebesgue measure on the unit circle which has radial part $\delta_1$.
 \end{itemize}
  This can be proved easily by computing the potentials of these measures and verifying that they satisfy the required variational conditions for the unconstrained
  equilibrium measure. Then next, we need to take care of the constraint: we have to sweep the unconstrained equilibrium measure under the constraint.
  This will be easy whenever the density $(tQ'(t))'$ is already under the constraint $\sigma'$, because then one only needs to sweep the dirac measure $\delta_1$
  under the constraint. This is achieved by taking the constraining density $\sigma'$ on $[a,1]$ with an appropriate $a$, as we have done for the case without
  external field. We will work this out explicitly in some of the examples in Section \ref{sec:ex} .
\end{enumerate}

\section{Proof of Theorem \ref{thm:2.1}}   \label{sec:4}

We will use the sets 
\[  E_n(\epsilon) = \{ q^{k/n} : k \in \mathbb{N} \textrm{ and } q^{k/n} \geq \epsilon \}.  \]
These are the points of the $q$-lattice after the scaling by taking $n$th roots, but away from the origin. 
We avoid the origin because it is an accumulation point of the $q$-lattice, and near the origin the points in $E_n=\{q^{k/n}; k=0,1,2,3,\ldots\}$
are not normal.
We will let $\epsilon \to 0$ in a specific way to conclude
that some limits can be carried over to the orthogonal polynomials on the full $q$-lattice. The main argument that assures that this can be done is that $0$ is not contained
in the support of the equilibrium measure $\mu^\sigma$ and therefore also not in the support of the radial part $\nu^\sigma$.
The sets $E_n(\epsilon)$ satisfy the conditions of Theorem \ref{thm:KVA}. Indeed, the first conditions follow since $E_n(\epsilon)$  is a finite set for every $\epsilon > 0$,
the second condition follows as we explained at the beginning of Section \ref{sec:3} and then restricting the sets in this construction 
to $[\epsilon,1]$. To verify
the third condition, we note that we have the inequality
\[   q^{k/n} - q^{(k+1)/n} = (1-q^{1/n}) q^{k/n} \geq \frac{1-q}{n} q^{k/n} \geq \frac{(1-q)\epsilon}{n}.  \]
Hence every point $x \in [\epsilon,1]$ is a normal point for $(E_n)_{n \in \mathbb{N}}$.

First we prove the following
\begin{lemma}  \label{lem:4.2}
For any $n \in \mathbb{N}$, let $P_n$ be the monic orthogonal polynomial of degree $n$ for the weight function $w$ on the $q$-lattice, and $R_{n,\epsilon}^*$ the monic
polynomial of degree $n$ minimizing $\|w^{1/2} R_n\|_{2,E_1(\epsilon)}$ over all monic polynomials $R_n$ of degree $n$. Let $(b_k)_{k \in \mathbb{N}}$ be a sequence
of strictly positive real numbers converging to $0$. Then
\begin{equation}  \label{limRnbkPn}
   \lim_{k \to \infty} ||w^{1/2} R_{n,b_k}^*\|_{2,E_1(b_k)} = \|w^{1/2} P_n\|_{2,q}.  
\end{equation}
Moreover
\begin{equation}  \label{RnPn}
     \lim_{k \to \infty} R_{n,b_k}^*(x) = P_n(x)  
\end{equation}
uniformly on $[0,1]$. 
\end{lemma} 

\begin{proof}
Fix $n \in \mathbb{N}$ and a sequence $(b_k)_{k\in \mathbb{N}}$ of strictly positive numbers that converges to $0$. We will first show that the sequence $(a_k)_{k \in \mathbb{N}}$
with
\[   a_k = \| w^{1/2} R_{n,b_k}^*\|_{2,E_1(b_k)}^2 , \]
converges to $\|w^{1/2} P_n\|_{2,q}^2$. Let $j_k$ be the largest positive integer such that $b_k \leq q^{j_k}$. Then we have
\[    \|w^{1/2} R_{n,b_k}^*\|_{2,E_1(b_k)}^2 = \sum_{i=0}^{j_k} (q^i-q^{i+1}) w(q^i) \bigl( R_{n,b_k}^*(q^i) \bigr)^2   \]
and
\[    \|w^{1/2} P_n\|_{2,q}^2 = \sum_{i=0}^\infty (q^i - q^{i+1}) w(q^i) P_n^2(q^i).  \]
The polynomial $R_{n,b_k}^*$ minimizes the norm $\|w^{1/2} R_n\|_{2,E(b_k)}$ among all monic polynomials of degree $n$, hence
\begin{equation}   \label{RnleqPn}
  0 \leq \|w^{1/2}R_{n,b_k}^*\|_{2,E_1(b_k)}^2 \leq \|w^{1/2} P_n\|_{2,E_1(b_k)}^2 \leq \|w^{1/2} P_n\|_{2,q}^2, 
\end{equation}
so that $(a_k)_{k \in \mathbb{N}}$ is a bounded sequence. Furthermore
\[    \|w^{1/2} R_{n,b_{k+1}}^*\|_{2,E_1(b_{k+1})}^2 \geq \|w^{1/2} R_{n,b_{k+1}}^*\|_{2,E_1(b_k)}^2 \geq  \|w^{1/2} R_{n,b_{k}}^*\|_{2,E_1(b_k)}^2, \]
so that $(a_k)_{k \in \mathbb{N}}$ is an increasing sequence. Therefore the sequence converges. We will show that it converges to $\|w^{1/2} P_n\|_{2,q}^2$.
The monic orthogonal polynomial $P_n$ minimizes the $L_2$-norm among all monic polynomials of degree $n$, hence
\begin{equation}  \label{PnRn}
    \|w^{1/2} P_n\|_{2,q}^2 \leq \|w^{1/2} R_{n,b_k}^*\|_{2,q}^2  = a_k + \sum_{i=j_k+1}^\infty (q^i-q^{i+1}) w(q^i) \bigl( R_{n,b_k}^* \bigr)^2. 
\end{equation}
The family $\{R_{n,b_k}^*, k \in \mathbb{N}\}$ with $n$ fixed is uniformly bounded on $[0,1]$ since they are monic polynomials with $n$ zeros in $[0,1]$, hence
by taking the limit for $k \to \infty$  in \eqref{PnRn} we see that $\lim_{k \to \infty} a_k \geq \|w^{^1/2} P_n\|_{2,q}^2$, which together with the converse inequality \eqref{RnleqPn} proves \eqref{limRnbkPn}.

Since $\{R_{n,b_k}^*, k \in \mathbb{N}\}$ with $n$ fixed is uniformly bounded on $[0,1]$, there exists a subsequence that converges to a monic polynomial $Q_n$.
By taking the limit along that subsequence in \eqref{RnleqPn} we see that
\[   \|w^{1/2} Q_n\|_{2,q}^2 \leq \|w^{1/2} P_n \|_{2,q}^2, \]
but since $P_n$ minimizes the $L_2$-norm, it follows that $Q_n=P_n$, and hence every converging subsequence has the same limit and \eqref{RnPn} follows.
\end{proof}

We will now show that the norm with weight $w$ and the norm with $w=1$ are comparable whenever \eqref{wzero} holds.
\begin{lemma}  \label{lem:4.3}
Suppose that \eqref{wzero} holds, i.e.,
\[    \lim_{n \to \infty} \frac{1}{n^2} \log w(x^n) = 0, \]
uniformly for $x$ in closed subsets of $(0,1]$. 
Then for any $\epsilon \in (0,1)$ and for any sequence $(R_n)_{n \in \mathbb{N}}$ of polynomials of degree $\leq n$ and $R_n \not\equiv 0$, one has
\[    \lim_{n \to \infty}  \left( \frac{\|w^{1/2}(x^n) R_n(x^n)\|_{2,E_n(\epsilon)}}{\|R_n(x^n)\|_{2,E_n(\epsilon)}} \right)^{1/n^2} = 1 .  \]
\end{lemma}

\begin{proof}
Fix $\epsilon \in (0,1)$ and a sequence of polynomials $(R_n)_{n \in \mathbb{N}}$ with $\deg R_n \leq n$ and $R_n \not\equiv 0$. We will first prove that
\begin{equation} \label{ratio<}
   \limsup_{n \to \infty} \left( \frac{\|w^{1/2}(x^n) R_n(x^n)\|_{2,E_n(\epsilon)}}{\|R_n(x^n)\|_{2,E_n(\epsilon)}} \right)^{1/n^2} \leq 1.  
\end{equation}
Observe that
\[    \|w^{1/2}(x^n)R_n(x^n)\|_{2,E_n(\epsilon)} \leq \|w^{1/2}(x^n)\|_{E_n(\epsilon)} \|R_n(x^n)\|_{2,E_n(\epsilon)}, \]
so that
\[   \left( \frac{\|w^{1/2}(x^n) R_n(x^n)\|_{2,E_n(\epsilon)}}{\|R_n(x^n)\|_{2,E_n(\epsilon)}} \right)^{1/n^2} \leq \|w^{1/2}(x^n)\|_{E_n(\epsilon)}^{1/n^2}
     \leq \|w^{1/2n^2}(x^n) \|_{E_n(\epsilon)} , \]
from which we find \eqref{ratio<}.

To prove the inequality in the other way, choose any $\delta >0$, then by assumption \eqref{wzero} there exists $N \in \mathbb{N}$ such that for $n \geq N$ one has
$w^{1/2n^2}(x_n) \geq 1-\delta$ for all $x \in [\epsilon,1]$. Since $w$ is only defined on the lattice points, we can assume without loss of generality that
it extends to a positive and continuous function on $[\epsilon,1]$. The function $w^{1/2}(x^n)$ then attains its minimum on $[\epsilon,1]$, and let's call this $m_n$.
Then for $n \geq N$ we have $m_n^{1/n^2} \geq 1-\delta$. Hence for $n \geq N$
\[   \left(  \frac{\|w^{1/2}(x^n) R_n(x^n) \|_{2,E_n(\epsilon)}}{\|R_n(x^n)\|_{2,E_n(\epsilon)}} \right)^{1/n^2}
     \geq \left( \frac{m_n \|R_n(x^n)\|_{2,E_n(\epsilon)}}{\|R_n(x^n)\|_{2,E_n(\epsilon)}} \right)^{1/n^2} \geq 1-\delta.  \]
Taking $\liminf$ then gives for every $\delta >0$
\[   \liminf_{n \to \infty}      \left(  \frac{\|w^{1/2}(x^n) R_n(x^n) \|_{2,E_n(\epsilon)}}{\|R_n(x^n)\|_{2,E_n(\epsilon)}} \right)^{1/n^2} \geq 1-\delta, \]
and for $\delta \to 0$ this gives the required inequality. Combining both inequalities then gives the desired result.
\end{proof} 

The result in Lemma \ref{lem:4.3} is given for the $L_2$-norm but in fact holds for every $L_p$-norm with $1 \leq p \leq \infty$, since we did not use any property
that holds only for $p=2$. Our preference for $p=2$ comes from our main interest in orthogonal polynomials on the $q$-lattice.

We can now prove our Theorem \ref{thm:2.1} by using the reasoning from \cite{Kuijl}, with a few modifications for our situation.

\begin{proof}[Proof of Theorem \ref{thm:2.1}]
The polynomial $P_n=p_n/\gamma_n$ is the monic polynomial of degree $n$ that minimizes the norm $\|w^{1/2}Q_n\|_{2,q}$ over monic polynomials
$Q_n$ of degree $n$, which is the same as saying that
\[   \|w^{1/2} P_n \|_{2,E_1}^{1/n^2} = \min_{Q_n(z)=z^n +\cdots} \|w^{1/2} Q_n(z)\|_{2,E_1}^{1/n^2} .  \]
Consider now the sequence $(b_k)_{k \in \mathbb{N}}$ defined by $b_k = 1/k^n$, then $b_k \to 0$ as $k \to \infty$ and $b_k >0$. So we can apply 
Lemma \ref{lem:4.2} to find
\[   \lim_{k \to \infty} \| w^{1/2} R_{n,b_k}^*\|_{2,E_1(b_k)}^{1/n^2} = \frac{1}{\gamma^{1/n^2}} \| w^{1/2} p_n \|_{2,E_1}^{1/n^2}.  \]
We claim that for every $n \in \mathbb{N}$
\[  \|w^{1/2}(x) R_{n,b_k}^*(x)\|_{2,E_1(b_k)} = \| w^{1/2}(x^n) R_{n,b_k}^*(x^n)\|_{2,E_n(\sqrt[n]{b_k})} . \]
This follows easily by the definition of the $L_2$-norms on these discrete sets. By our choice of $(b_k)_k$ we then have
\[  \|w^{1/2}(x) R_{n,b_k}^*(x)\|_{2,E_1(b_k)} = \| w^{1/2}(x^n) R_{n,b_k}^*(x^n)\|_{2,E_n(1/k)}, \]
and by Lemma \ref{lem:4.3}  we then have
\[   \lim_{n \to \infty} \|w^{1/2}(x)R_{n,b_k}^*(x)\|_{2,E_1(b_k)}^{1/n^2} = \lim_{n \to \infty} \|R_{n,b_k}^*(x^n)\|_{2,E_n(1/k)}^{1/n^2} . \]
From \cite[Lemma 8.3~(a)]{KuijlWVA} it follows that the $n$-th root of the $L_2$-norm on discrete sets is comparable with the $n$-th root of the sup-norm on the same set. Since $R_{n,b_k}(x^n)$ is a polynomial of degree $n^2$, we thus have
\[   \lim_{n \to \infty} \|w^{1/2}(x)R_{n,b_k}^*(x)\|_{2,E_1(b_k)}^{1/n^2} = \lim_{n \to \infty} \|R_{n,b_k}^*(x^n)\|_{\infty,E_n(1/k)}^{1/n^2} . \]
Note that $R_{n,b_k^*}$ is the polynomial that minimizes the $L_2$-norm on $E_1(b_k)$, therefore we can replace the polynomials
$R_{n,b_k}^*$ on the right hand side by the monic polynomials $P_{n,b_k}^*$ of degree $n$ that minimize the sup-norm on $E_1(b_k)$, giving
\[  \lim_{n \to \infty}  \|w^{1/2}(x)R_{n,b_k}^*(x)\|_{2,E_1(b_k)}^{1/n^2} = \lim_{n \to \infty} \|P_{n,b_k}^*(x^n)\|_{\infty,E_n(1/k)}^{1/n^2} . \]
Take now $P_n(x) = (x-x_{1,n})(x-x_{2,n})\cdots(x-x_{n,n})$ and assume that the $n$-th roots of these zeros are distributed according to the radial part
of $\mu^{\sigma_{1/k}}$, where $\sigma_{1/k}$ is the constraint restricted to the interval $[1/k,1]$. Observe that $\sigma_{1/k}$ is an admissible
constraint for $k$ large enough. Then the zeros of $P_n(x^n)$ are distributed according to $\mu^{\sigma_{1/k}}$ and the normalized 
zero counting measure for $P_n(x^n)$
\[  \xi_{n} = \frac{1}{n^2} \sum_{k=1}^n \sum_{j=0}^{n-1} \delta_{x_{k,n}^{1/n} \omega_n^j}, \qquad   \omega_n = e^{2\pi i/n}, \]
converges in weak-$\ast$ sense to $\mu_{\sigma_{1/k}}$. Then by \cite[Lemma 4~(e)]{Kuijl} it follows that
\[  \lim_{n \to \infty} U(z;\xi_{n}) = U(z;\mu^{\sigma_{1/k}}), \]
uniformly on  compact subsets of $\overline{\mathbb{D}}\setminus \{0\}$. 
Furthermore $|P_n(x^n)|$ assumes its maximum in $\textrm{supp}(\mu^{\sigma_{1/k}})$ due to the variational conditions. Hence we can take for any 
$n \geq 1$ a point $x_n \in \textrm{supp}(\mu^{\sigma_{1/k}})$ so that $|P_n(x^n)|$ attains its maximum at $x_n$. Obviously $(x_n)_n$ is a bounded
sequence and hence it has a convergent subsequence. Suppose that $x_n \to x^*$ along such a subsequence, then $x^* \in \textrm{supp}(\mu^{\sigma_{1/k}})$
and from the variational conditions \eqref{var1}--\eqref{var2} we have that $U(x^*;\mu^{\sigma_{1/k}}) = w^{\sigma_{1/k}}$.
Then by the principle of descent \cite[Thm. 6.8 in Ch.~I]{SaffTotik}, combined with the fact that $|P_n(x^n)|=\exp \bigl(-n^2 U(x:\xi_n)\bigr)$,
it follows that
\begin{eqnarray*}
   \limsup_{n \to \infty} \|P_n(x^n)\|_{\infty,E_n(1/k)}^{1/n^2}  &\leq& \lim_{n \to \infty} \|P_n(x^n)\|_{\infty,\overline{\mathbb{D}}}^{1/n^2} \\ 
   &\leq&   \exp \left( -U(x^*;\mu^{\sigma_{1/k}} \right) =  \exp (-  w^{\sigma_{1/k}}) .  
\end{eqnarray*}
From \cite[Corollary 7.3]{KuijlWVA} it follows that
\[    \liminf_{n \to \infty}  \|P_n(x^n)\|_{\infty,E_n(1/k)}^{1/n^2} \geq \exp( -  w^{\sigma_{1/k}}), \]
so that we can conclude that
\[   \lim_{n \to \infty}  \|P_n(x^n)\|_{\infty,E_n(1/k)}^{1/n^2} = \exp( -  w^{\sigma_{1/k}}).  \]
By the minimality of $P_{n,b_k}^*$ it follows that
\[  \lim_{n \to \infty}  \|P_{n,b_k}^*(x^n)\|_{\infty,E_n(1/k)}^{1/n^2} 
    \leq \lim_{n \to \infty}  \|P_n(x^n)\|_{\infty,E_n(1/k)}^{1/n^2} = \exp( -  w^{\sigma_{1/k}}). \]
Applying \cite[Corollary 7.3]{KuijlWVA} once more we find
\begin{equation}  \label{P*wk}
   \lim_{n \to \infty}  \|P_{n,b_k}^*(x^n)\|_{\infty,E_n(1/k)}^{1/n^2}     = \exp( -  w^{\sigma_{1/k}}).  
\end{equation}
Since $0 \notin \textrm{supp}(\nu^\sigma)=[q,1]$, it follows that the sequence $(w^{\sigma_{1/k}})_k$ is constant for $k$ large enough. 
Taking the limit for $k \to \infty$ in \eqref{P*wk} gives $\exp(-w^\sigma)$ on the right hand side, and on the left hand side
the limit for $n \to \infty$ and $k \to \infty$ can be interchanged since 
$1/k \notin \textrm{supp}(\nu^\sigma)$ for $k$ large enough, and hence $\nu^\sigma|_{[1/k,1]} = \nu^\sigma$. 
Hence
\[   \lim_{n \to \infty} \frac{1}{\gamma_n^{1/n^2}} \| w^{1/2}(x^n) p_n(x^n)\|_{2,E_n}^{1/n^2} =
      \lim_{k \to \infty} \lim_{n \to \infty} \|P_{n,b_k}^*(x^n)\|_{\infty,E_n(1/k)}^{1/n^2} = \exp(-w^\sigma), \]
and  $w^\sigma= -\frac12 \log q$, as is evident from the potential $U(x;\mu^\sigma)$ which was computed at the end of Section \ref{sec:3.1}.
This proves the asymptotic behavior \eqref{normq} for the norm $\gamma_n$ since $\| w^{1/2}(x^n) p_n(x^n)\|_{2,E_n}=1$.
The result about the zero distribution now simply follows from Theorem \ref{thm:KVA}. 
\end{proof}

\section{Proof of Theorem \ref{thm:2.2}}  \label{sec:5}

The proof for the case with an admissible external field $Q$ is very similar. The only important difference is that we have to make sure that
$0$ is not in the support of the equilibrium measure $\mu_w^\sigma$. For $Q=0$ this was clear since the support of $\mu^\sigma$ in given
by the annulus $\{ q \leq |z| \leq 1\}$, but for an external field we need an additional assumption on $Q$. A sufficient condition is
that $Q$ is decreasing on $(0,\epsilon)$ for some $\epsilon >0$. Indeed, suppose that on the contrary $\mu_w^\sigma$ has positive mass $M>0$
on $\{ |z| \leq \delta \}$, then consider the circular measure $\hat{\lambda}$ with radial part $\lambda$ given by
\[   \lambda = \nu_w^\sigma|_{[\epsilon,1]} + M \delta_{\epsilon}, \]
then it is easy to see that
\[   U(z;\mu_w^\sigma) - U(z;\hat{\lambda}) \geq 0, \qquad z \in \overline{\mathbb{D}},  \]
and
\[    \int Q(x) \, d\mu_w^\sigma(x) - \int Q(x)\, d\hat{\lambda}(x) \geq 0, \]
and since the logarithmic potential of a circular measure is decreasing on $[0,\infty)$ (see, e.g., \cite[Lemma 4 (c)]{Kuijl}) one 
has that
\[  I(\mu_w^\sigma) - I(\lambda) = \int U(x;\mu_w^\sigma)\, d\mu_w^\sigma(x) - \int U(x;\hat{\lambda})\, d\hat{\lambda}(x) \geq 0, \]
so that $I_w(\mu_w^\sigma) \geq I_w(\hat{\lambda})$, which violates the minimality of $\mu_w^\sigma$ for the unconstrained case.
The constrained case can be shown in a similar way if one replaces the $\delta_\epsilon$ by a suitable restriction 
of the constraint $\sigma$ to the interval $[a,\epsilon]$ with $a>0$.

The extension of Lemma \ref{lem:4.3} for an external field is
\begin{lemma} \label{lem:5.1}
Suppose that
\[   \lim_{n \to \infty} \frac{1}{n^2} \log w(x^n) = -2Q(x), \]
uniformly on compact subsets of $(0,1]$. Then for all $\epsilon \in (0,1)]$ and for any sequence $(R_n)_{n \in \mathbb{N}}$ of polynomials
of degree $\leq n$ and $R_n \not\equiv 0$, one has
\[  \lim_{n \to \infty}  \left( \frac{\|w^{1/2}(x^n) R_n(x^n)\|_{2,E_n(\epsilon)}}
    {\|\exp(-n^2Q(x))R_n(x^n)\|_{2,E_n(\epsilon)}} \right)^{1/n^2} = 1 .  \] 
\end{lemma}

\begin{proof}
We claim that it is sufficient to prove that
\begin{equation}  \label{ratioinf}
\lim_{n \to \infty}  \left( \frac{\|w^{1/2}(x^n) R_n(x^n)\|_{\infty,E_n(\epsilon)}}{\|\exp(-n^2Q(x))R_n(x^n)\|_{\infty,E_n(\epsilon)}} \right)^{1/n^2} = 1 . 
\end{equation}
Indeed, if \eqref{ratioinf} holds, then we can use \cite[Lemma 8.3 (b)]{KuijlWVA} for both $w^{1/2}(x^n)$ and $\exp(-n^2 Q(x))$. We know already
that the constraint is admissible and that the sets $E_n(\epsilon)$ satisfy the conditions of Theorem \ref{thm:2.1}. We thus have
\[   \lim_{n \to \infty}  \left( \frac{\|w^{1/2}(x^n) R_n(x^n)\|_{2,E_n(\epsilon)}}{\|w(x^n)R_n(x^n)\|_{\infty,E_n(\epsilon)}} \right)^{1/n^2} = 1 , \]
and
\[ \lim_{n \to \infty}  \left( \frac{\|exp(-n^2Q(x))R_n(x^n)\|_{2,E_n(\epsilon)}}{\|\exp(-n^2Q(x))R_n(x^n)\|_{\infty,E_n(\epsilon)}} \right)^{1/n^2} = 1 , \]
so that the result of the lemma indeed follows.

Choose any $\delta>0$, then for $n$ large enough we have that $\exp(-\delta-Q(x))^{n^2} < w^{1/2}(x^n) < \exp(\delta-Q(x))^{n^2}$ for all
$x \in [\epsilon,1]$. Hence we have for $n$ large enough
\[  \left( \frac{\|w^{1/2}(x^n) R_n(x^n)\|_{\infty,E_n(\epsilon)}}{\|\exp(-n^2Q(x))R_n(x^n)\|_{\infty,E_n(\epsilon)}} \right)^{1/n^2}
   \leq  \left( \frac{\|\exp(\delta-Q(x))^{n^2} R_n(x^n)\|_{\infty,E_n(\epsilon)}}{\|\exp(-n^2Q(x))R_n(x^n)\|_{\infty,E_n(\epsilon)}} \right)^{1/n^2}
   = \exp(\delta). \]
Since this holds for any $\delta>0$, we can let $\delta \to 0$ to find
\[  \limsup_{n \to \infty} \left( \frac{\|w^{1/2}(x^n) R_n(x^n)\|_{\infty,E_n(\epsilon)}}
{\|\exp(-n^2Q(x))R_n(x^n)\|_{\infty,E_n(\epsilon)}} \right)^{1/n^2} \leq 1.  \]
In a similar way
\begin{eqnarray*}  
 \left( \frac{\|w^{1/2}(x^n) R_n(x^n)\|_{\infty,E_n(\epsilon)}}{\|\exp(-n^2Q(x))R_n(x^n)\|_{\infty,E_n(\epsilon)}} \right)^{1/n^2}
  &\geq & \left( \frac{\|\exp(-\delta-Q(x))^{n^2} R_n(x^n)\|_{\infty,E_n(\epsilon)}}{\|\exp(-n^2Q(x))R_n(x^n)\|_{\infty,E_n(\epsilon)}} \right)^{1/n^2} \\
  & = & \exp(-\delta), 
\end{eqnarray*}
and by letting $\delta \to 0$ we have
\[  \liminf_{n \to \infty} \left( \frac{\|w^{1/2}(x^n) R_n(x^n)\|_{\infty,E_n(\epsilon)}}
    {\|\exp(-n^2Q(x))R_n(x^n)\|_{\infty,E_n(\epsilon)}} \right)^{1/n^2}   \geq 1.  \]
Together both inequalities gives the desired result \eqref{ratioinf}.  
\end{proof}

Lemma \ref{lem:5.1} is formulated for the $L_2$-norm and we actually prove it for the $L_\infty$-norm, but it holds for every $L_p$-norm
with $1 \leq p \leq \infty$. 

\begin{proof}[Proof of Theorem \ref{thm:2.2}]
Recall that the monic polynomial $P_n(x)=p_n(x)/\gamma_n$ minimizes the $L_2$-norm $\|Q_n\|_{2,q}$ among all monic polynomials $Q_n$ of degree $n$.
Consider the sequence $(b_k)_{k\in \mathbb{N}}$ with 
\[  b_k = \left( \frac{1}{k} \right)^n, \]
then $b_k >0$ and $b_k \to 0$. We can then apply Lemma \ref{lem:4.2} to find
\[  \lim_{k \to \infty} \|w^{1/2}R_{n,b_k}^*\|_{2,E_1(b_k)}^{1/n^2} = \frac{1}{\gamma_n^{1/n^2}} \| w^{1/2} p_n \|^{1/n^2} = \frac{1}{\gamma_n^{1/n^2}}.
\]
From the definition of the $L_2$-norms, it easily follows that
\[   \|w^{1/2} R_{n,b_k}^* \|_{2,E_1(b_k)} = \| w^{1/2}(x^n) R_{n,b_k}^*(x^n)\|_{2,E_n(1/k)}.  \]
By Lemma \ref{lem:5.1} it follows that
\[  \lim_{n \to \infty}  \|w^{1/2} R_{n,b_k}^* \|_{2,E_1(b_k)}^{1/n^2} = \lim_{n \to \infty} \|\exp(-n^2Q(x)) R_{n,b_k}^*(x^n)\|_{2,E_n(1/k)}^{1/n^2}. \]
From \cite[Lemma 8.3 (a)]{KuijlWVA} it follows that the $n$-th root of the $L_2$-norm on the discrete sets is comparable with the $n$-th root of the
sup-norm on the same set, and since $R_{n,b_k}^*(x^n)$ is a polynomial of degree $n^2$, this gives
\[  \lim_{n \to \infty}  \|w^{1/2} R_{n,b_k}^* \|_{2,E_1(b_k)}^{1/n^2} 
= \lim_{n \to \infty} \|\exp(-n^2Q(x)) R_{n,b_k}^*(x^n)\|_{\infty,E_n(1/k)}^{1/n^2}. \]
The polynomial $R_{n,b_k}^*$ on the left side minimizes the $L_2$-norm on $E_1(b_k)$, hence we can replace it on the right hand side
by $P_{n,b_k}^*$ which minimizes the $L_\infty$-norm on $E_1(b_k)$ and retain equality:
\[  \lim_{n \to \infty}  \|w^{1/2} R_{n,b_k}^* \|_{2,E_1(b_k)}^{1/n^2} 
= \lim_{n \to \infty} \|\exp(-n^2Q(x)) P_{n,b_k}^*(x^n)\|_{\infty,E_n(1/k)}^{1/n^2}. \]
From here on the proof continues exactly as the proof of Theorem \ref{thm:2.1}, except we use the variational conditions \eqref{Qvar1}--\eqref{Qvar2}
instead of \eqref{var1}--\eqref{var2}.
\end{proof} 

\section{Examples}  \label{sec:ex}
In this section we apply Theorem \ref{thm:2.1} and Theorem \ref{thm:2.2} to three families of classical orthogonal polynomials on the $q$-lattice.

\subsection{Little $q$-Laguerre polynomials}
Little $q$-Laguerre polynomials (or Wall polynomials) are given by \cite[\S 14.20]{KLS}
\[  p_n(x;a|q) = {}_2\phi_1 \left( \left. \begin{array}{c} q^{-n}, 0 \\ aq \end{array} \right| q, qx \right). \]
For $0 < aq < 1$ the orthogonality relations are
\[   \sum_{k=0}^\infty p_n(q^k;a|q)p_m(q^k;a|q) \frac{(aq)^k}{(q;q)_k} = \frac{(aq)^n}{(aq;q)_\infty} \frac{(q;q)_n}{(aq;q)_n} \delta_{n,m} . \] 
The weights are thus given by
\[     w_k = \frac{a^k}{(q;q)_k} = w(q^k;a), \qquad  w(x;a) = (qx;q)_\infty x^\alpha, \]
where $q^\alpha = a$. Clearly we have
\[   \lim_{n \to \infty} \frac{1}{n^2} \log w(x^n;a) = 0, \qquad x \in [0,1], \]
so that Theorem \ref{thm:2.1} holds if we keep $a$ fixed. In particular one has
\[   \lim_{n \to \infty} \gamma_n^{1/n^2} = q^{-1/2}, \]
which can easily be verified by using the explicit expression
\[   \gamma_n(a) = q^{-\frac{n(n-1)}{2}} \frac{\sqrt{(aq;q)_\infty}}{\sqrt{(aq;q)_n (aq)^n (q;q)_n}} . \] 
The asymptotic distribution of the points $x_{1,n}^{1/n},\ldots,x_{n,n}^{1/n}$ is then given by the measure with density (see Fig. \ref{fig:2})
\[  (\nu^\sigma)'(t) = \frac{-1}{\log q} \frac{1}{t}, \qquad t \in [q,1] .  \]
This means that the zeros are dense on $[q,1]$ and the constraint $\sigma$ holds on the full support of this equilibrium measure.

If we take $a=q^{2n\alpha}$ with $\alpha > 0$ and consider the polynomials $p_n(x;q^{2n\alpha}|q)$, then
\[   \lim_{n \to \infty} - \frac{1}{n^2} \log w(x^n,q^{2n\alpha}|q) = -2 \alpha \log x, \qquad x \in (0,1], \]
so that we get an external field $Q(x) = -\alpha \log x$ on $(0,1]$. For $\alpha >0$ this is a decreasing function on $(0,1]$, which means that if we 
add $Q(|x|)$ to  \eqref{var1}--\eqref{var2} for the equilibrium measure then we get the variational inequalities 
\eqref{Qvar1}--\eqref{Qvar2} and the constant $w_Q^\sigma$ is $w^\sigma + Q(q)=-\frac12 \log q - \alpha \log q$.
Hence Theorem \ref{thm:2.2} gives
\[   \lim_{n \to \infty} \gamma_n^{1/n^2} = q^{-1/2-\alpha}, \]
and this can be easily verified using the explicit expression of $\gamma_n$ given above.
The asymptotic distribution of the points  $x_{1,n}^{1/n},\ldots,x_{n,n}^{1/n}$ is still the same measure $\nu^\sigma$.

\subsection{$q$-Bessel polynomials}

The $q$-Bessel polynomials are given by \cite[\S 14.22]{KLS}
\[    y_n(x;a;q) = {}_2\phi_1 \left( \left. \begin{array}{c} q^{-n}, -aq^{n} \\ 0 \end{array} \right| q, qx \right), \]
where $a >0$. Their orthogonality is given by
\[    \sum_{k=0}^\infty \frac{a^k}{(q;q)_k} q^{\binom{k}{2}} q^k y_n(q^k;a;q) y_m(q^k;a;q) 
     = (q;q)_n (-aq^n;q)_\infty \frac{(aq)^n q^{\binom{n}{2}}}{(1+aq^{2n})} \delta_{n,m}, \]
so that the leading coefficient of the orthonormal polynomial is given by
\[  \gamma_n = q^{-\frac{3n(n-1)}{4}} (-a;q)_{2n} \frac{\sqrt{1+aq^{2n}}}{\sqrt{(q;q)_n(-a;q)_\infty (aq)^n(-a;q)_n}} . \]
The weights are
\[   w_k = q^{\frac{k(k-1)}{2}} \frac{(aq)^k}{(q;q)_k} = w(q^k), \]
where
\[    w(x) = \exp \left( \frac{(\log x)^2}{2 \log q} \right) x^{\alpha+1/2} \frac{(qx;q)_\infty}{(q;q)_\infty}. \]
In this case
\[    \lim_{n \to \infty} \frac{-1}{n^2} \log w(x^n) =  - \frac{(\log x)^2}{2 \log q}, \]
so that
\[   Q(x) =  - \frac{(\log x)^2}{4 \log q}. \]
This external field $Q$ is decreasing on $(0,1]$ so that the constrained equilibrium measure is again
$\nu^\sigma$ and the equilibrium constant is $w_Q^\sigma = w^\sigma + Q(q)= -\frac34 \log q$.
Theorem \ref{thm:2.2} then gives
\[  \lim_{n \to \infty} \gamma_n^{1/n^2} = q^{-3/4}, \]
which can indeed be verified from the explicit expression for $\gamma_n$. The asymptotic distribution of 
the $x_{1,n}^{1/n},\ldots,x_{n,n}^{1/n}$ is again given by the measure $\nu^\sigma$.

One could also take $a=q^{2n\alpha}$. The external field then becomes
\begin{equation}  \label{eq:Qex}
     Q(x) = - \frac{(\log x)^2}{4\log q} - \alpha \log x, \qquad x \in (0,1]. 
\end{equation}
For $\alpha > 0$ one still does not change the asymptotic distribution of the
scaled zeros, but the equilibrium constant becomes $w_Q^\sigma = w^\sigma - \frac14 \log q -\alpha \log q=
(-\frac34 -\alpha) \log q$. Theorem \ref{thm:2.2} then gives
\[    \lim_{n \to \infty} \gamma_n^{1/n^2} = q^{-\alpha-\frac34} .  \]
For $\alpha < 0$ the situation changes because the external field $Q$ in \eqref{eq:Qex} is no longer decreasing on $(0,1]$, but only decreasing on $(0,q^{-2\alpha})$. 
But then we can use Theorem \ref{thm:ST} and adapt it to our situation on $(0,1]$ and constraint $\sigma$. 
Observe that 
\[    Q'(r) = - \frac{\log r}{2\log q} \frac{1}{r} - \frac{\alpha}{r} , \qquad 0 < r \leq 1 \]
so that $rQ'(r)$ is increasing on $(0,1]$ for every $\alpha \in \mathbb{R}$. We see that $Q'(r) >0$ for $r > q^{-2\alpha}$ and that
$RQ'(R)=1$ for $R=q^{-2(1+\alpha)}$. If $R \leq 1$ (i.e., $\alpha \leq -1$) the unconstrained equilibrium measure is then given by the circular symmetric measure $\mu_Q$ with
radial component $\nu_Q$ that has the density
\[    \nu_Q'(r) = (rQ'(r))' = \frac{-1}{2 \log q} \frac{1}{r}, \qquad   r \in [q^{-2\alpha}, q^{-2(1+\alpha)}]. \]
Observe that $\nu_Q' \le \sigma'$ so that this unconstrained equilibrium measure is also
the solution of the constrained equilibrium problem. 
The potential of this circular symmetric measure with radial part $\nu_Q$ is given by
\[   U(z;\mu_Q) = \begin{cases} (2\alpha+1) \log q, & |z| < q^{-2\alpha} , \\[6pt]
                            \ds    \frac{(\log |z|)^2}{4 \log q} + \alpha \log |z| + (\alpha+1)^2 \log q, & q^{-2\alpha} < |z| < q^{-2(\alpha+1)}, \\[10pt]
                                - \log |z|, & |z| > q^{-2(\alpha+1)} ,
                  \end{cases}  \]
and then one easily verifies that the variational inequalities \eqref{Qvar1}--\eqref{Qvar2} hold. In fact, one has
\[      U(x;\mu_Q) + Q(x) = (\alpha+1)^2 \log q , \qquad  q^{-2\alpha} < x < q^{-2(\alpha+1)}, \]
so that the equilibrium constant is $w_Q^\nu = (\alpha+1)^2 \log q$, see Fig.~\ref{fig:4} on the left. It then follows from Theorem \ref{thm:2.2} that
\[   \lim_{n \to \infty}  \gamma_n^{1/n^2} = q^{(\alpha+1)^2}   .   \]  
For $-1 < \alpha < 0$ we see that $q^{-2(\alpha+1)} >1$ so there is no solution $R_0$ in $(0,1]$. In that case the unconstrained equilibrium problem is the circular symmetric measure with radial part $\nu_Q$ that has density
\begin{equation}   \label{Qdens}
   \frac{-1}{2 \log q} \frac{1}{r} , \qquad r \in [q^{-2\alpha},1], 
\end{equation}
together with a with a Dirac mass of size $1+\alpha$  at $r=1$. This can be verified by computing the potential and verifying the
variational conditions. This measure does not satisfy the constraint $\nu_Q \leq \sigma$, so one needs
to sweep the Dirac measure under the constraint measure $\sigma$. 
For $\alpha > -\frac12$ this will still give the measure with radial part $\nu^\sigma$.
For $-1 < \alpha < - \frac12$ we can replace the $(1+\alpha) \delta_1$ by part of the constraint $\sigma$, enough to add a mass of size $1+\alpha$ to
the density \eqref{Qdens}. This gives the measure $\nu_Q^\sigma$ with density
\begin{equation}  \label{Qdensig}    
   (\nu_Q^\sigma)'(r) =  \begin{cases}   \ds \frac{-1}{2\log q} \frac1r, &  r \in [q^{-2\alpha},q^{2(\alpha+1)}], \\[10pt]
                     \ds   \frac{-1}{\log q} \frac{1}{r}, & r \in [q^{2(\alpha+1)},1].
       \end{cases}  
\end{equation}
Indeed, the logarithmic potential of this measure is $U(z;\mu_Q^\sigma) = U(z;\mu_1) + U(z;\mu_2)$ where $\mu_1$ and $\mu_2$ are
circular symmetric measures with absolutely continuous radial parts $\nu_1$ and $\nu_2$ given by
\[   \nu_1'(r) =  \frac{-1}{2\log q} \frac1r, \qquad  r \in [q^{-2\alpha},1], \]
and
\[   \nu_2'(r) = \frac{-1}{2\log q} \frac1r,  \qquad  r \in [q^{2(\alpha+1)},1].  \]
One has
\[   U(z;\mu_1) = \begin{cases} -\alpha^2 \log q, & |z| < q^{-2\alpha}, \\[6pt]
                     \ds          \frac{(\log |z| )^2}{4 \log q} + \alpha \log |z|, & q^{-2\alpha} \leq |z| \leq 1, \\[10pt]
                                \alpha \log |z|, & |z| > 1, 
                  \end{cases}  \]
and
\[   U(z;\mu_2) = \begin{cases}  -(\alpha+1)^2 \log q, & |z| < q^{2(\alpha+1)}, \\[6pt]
                        \ds        \frac{(\log |z| )^2}{4\log q} - (\alpha+1) \log |z|, & q^{2(\alpha+1)} \leq |z| \leq 1, \\[10pt]
                                 -(\alpha+1) \log |z|, & |z| > 1, 
                  \end{cases}  \]
so that
\[   U(z;\mu) + Q(z) = \begin{cases}
               \ds    -(2\alpha^2+2\alpha+1) \log q - \frac{(\log |z| )^2}{4 \log q} - \alpha \log |z|, & |z| < q^{-2\alpha}, \\
                   -(\alpha+1)^2 \log q, & q^{-2\alpha} \leq |z| \leq q^{2(\alpha+1)}, \\[6pt]
               \ds    \frac{(\log |z|)^2}{4 \log q} - (\alpha+1) \log |z|, & q^{2(\alpha+1)} < |z| \leq 1.
                      \end{cases}  \]
This function is decreasing on $(0,q^{-2\alpha}]$, constant on $[q^{-2\alpha},q^{2(\alpha+1)}]$, and again decreasing on $[q^{2(\alpha+1)},1]$,
see Fig.~\ref{fig:4} on the right. Hence the variational inequalities \eqref{Qvar1}--\eqref{Qvar2} hold and the equilibrium constant $w_Q^\sigma$ is given by
$-(\alpha+1)^2 \log q$. 

\begin{figure}[ht]
\centering
\includegraphics[width=2.5in]{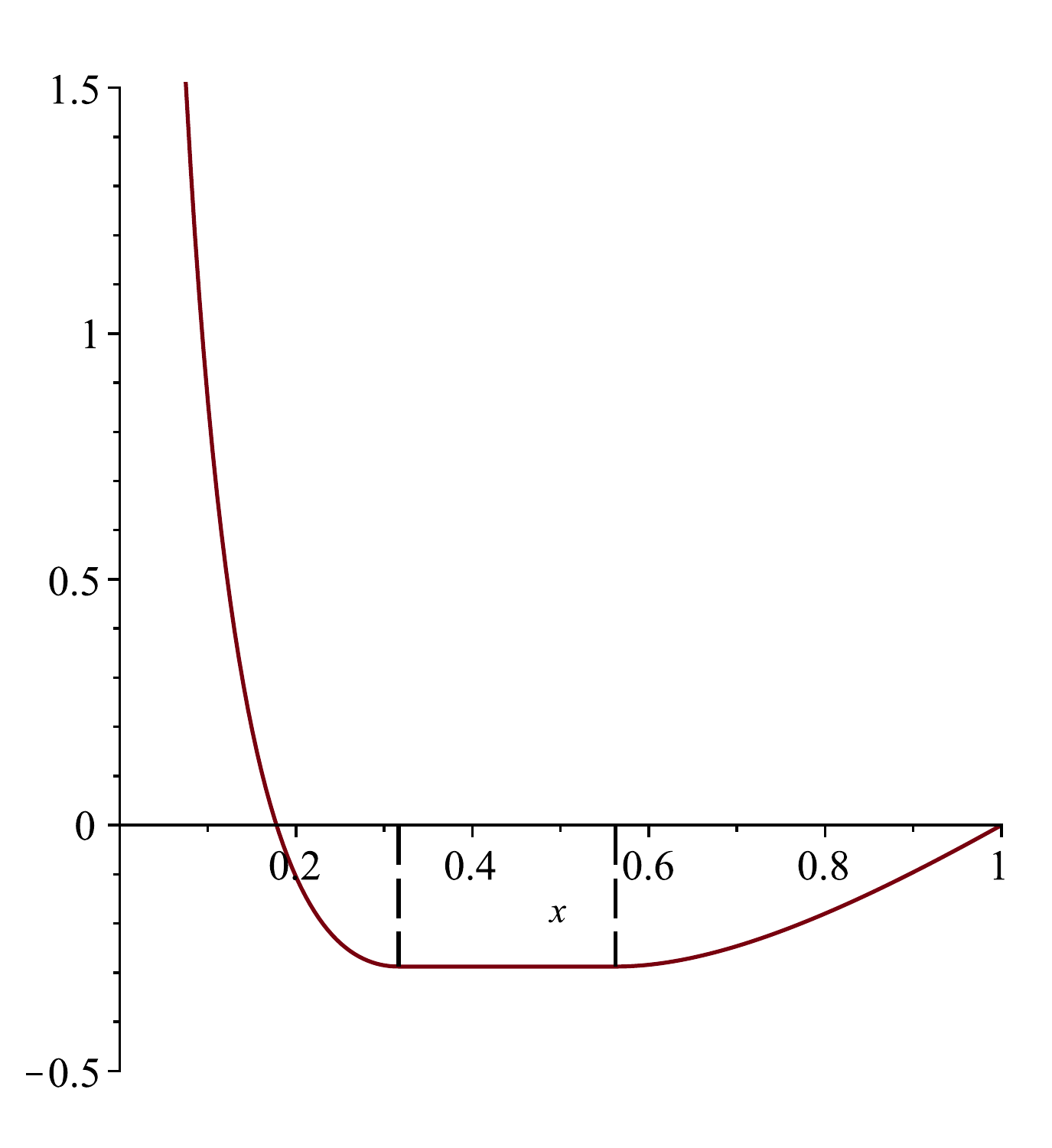}
\quad
\includegraphics[width=2.5in]{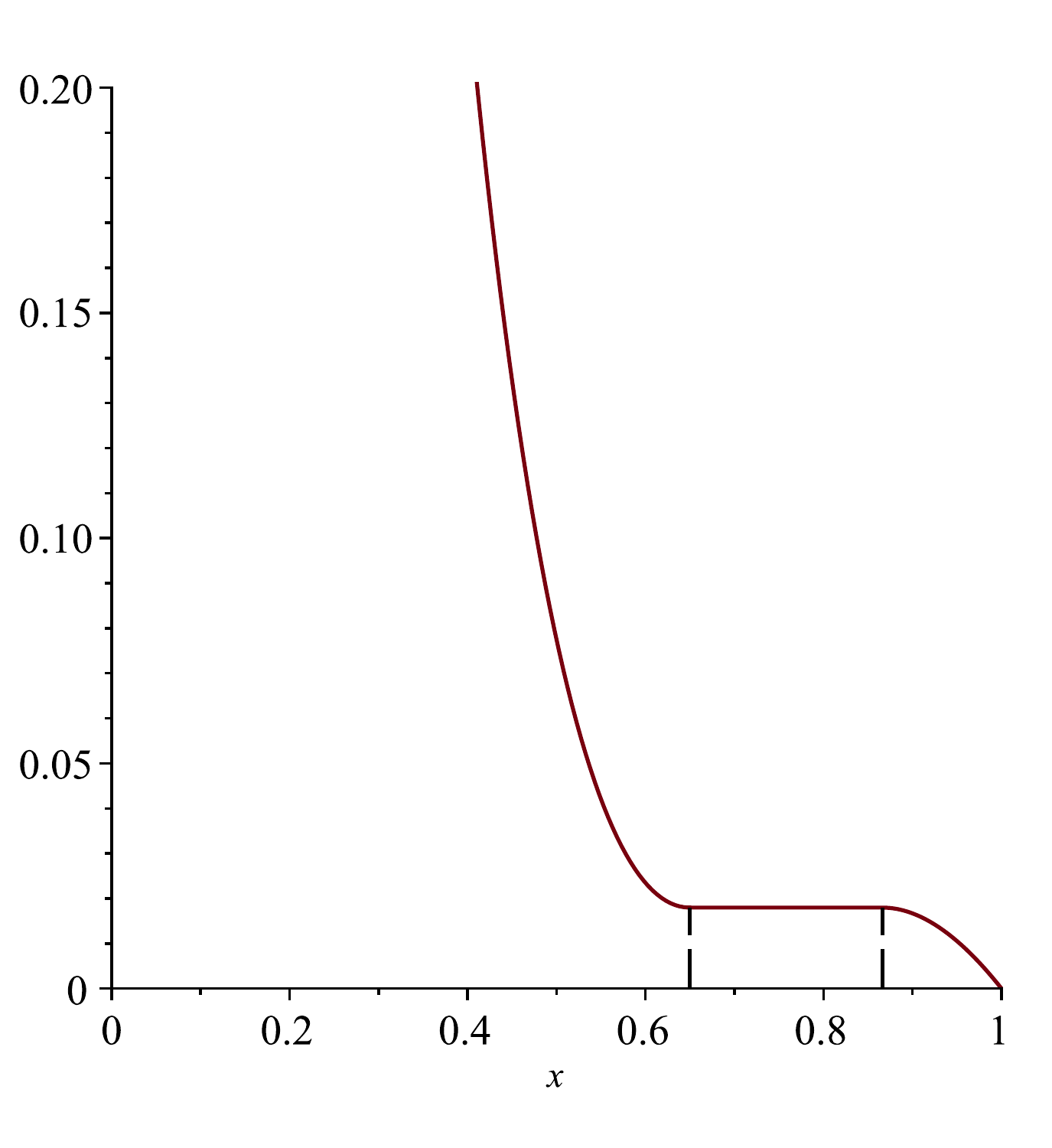}
\caption{The function $U(z;\mu_Q^\sigma)+Q(z)$ for $q=3/4$, $\alpha=-2$ (left) and $\alpha=-3/4$ (right).}
\label{fig:4}
\end{figure}

Summarizing, we thus have 

\begin{theorem} \label{thm:6.1}
For the $q$-Bessel polynomials with $a=q^{2n\alpha}$ one has
\[  \lim_{n \to \infty} \gamma_n^{1/n^2} = \begin{cases} q^{-\frac34 - \alpha}, & \textup{ if } \alpha \geq -\frac12, \\
                                                         q^{-(\alpha+1)^2}, & \textup{ if } -1 < \alpha < - \frac12, \\
                                                         q^{(\alpha+1)^2}, & \textup{ if } \alpha \leq -1.
                                           \end{cases}  \]
The asymptotic distribution of the zeros $x_{1,n}^{1/n},\ldots,x_{n,n}^{1/n}$ is given by
\[   \lim_{n \to \infty} \frac{1}{n} \sum_{k=1}^n f(x_{k,n}^{1/n}) = \int f(r) v(r)\, dr, \]
for every continuous function $f$ on $[0,1]$, where $v(r) = (\nu_Q^\sigma)'$ is given by
\[     v(r) = \begin{cases}  \ds \frac{-1}{\log q} \frac{1}{r} \chi_{[q,1]}, & \textup{ if } \alpha \geq - \frac12, \\[10pt]
                             \ds \frac{-1}{2\log q} \frac1r \chi_{[q^{-2\alpha},1]}  - \frac{1}{2\log q} \frac1r \chi_{[q^{2(\alpha+1)},1]}, &
                             \textup{ if } -1 < \alpha < -\frac12, \\[10pt]
                             \ds \frac{-1}{2\log q} \frac{1}{r} \chi_{[q^{-2\alpha},q^{-2(\alpha+1)}]}, & \textup{ if } \alpha \leq -1. 
                \end{cases}  \]
\end{theorem}

We have plotted the density $(\nu_Q^\nu)'$ in Fig. \ref{fig:5} for $\alpha = -2$ (left) and $\alpha=-3/4$ (right). For $\alpha \geq -1/2$
the density is given in Fig. \ref{fig:2}.

\begin{figure}[ht]
\centering
\includegraphics[width=2.5in]{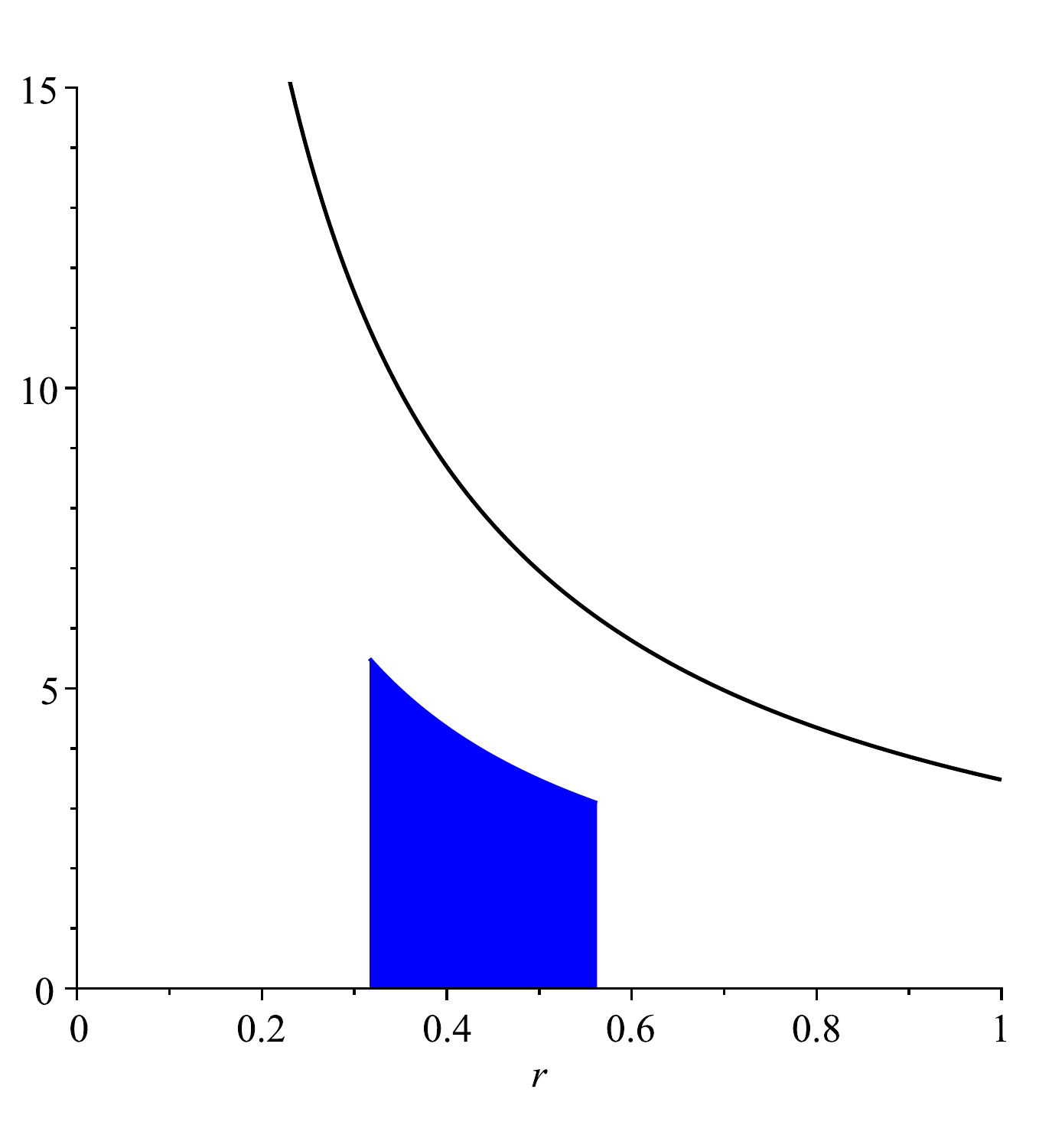}
\quad
\includegraphics[width=2.5in]{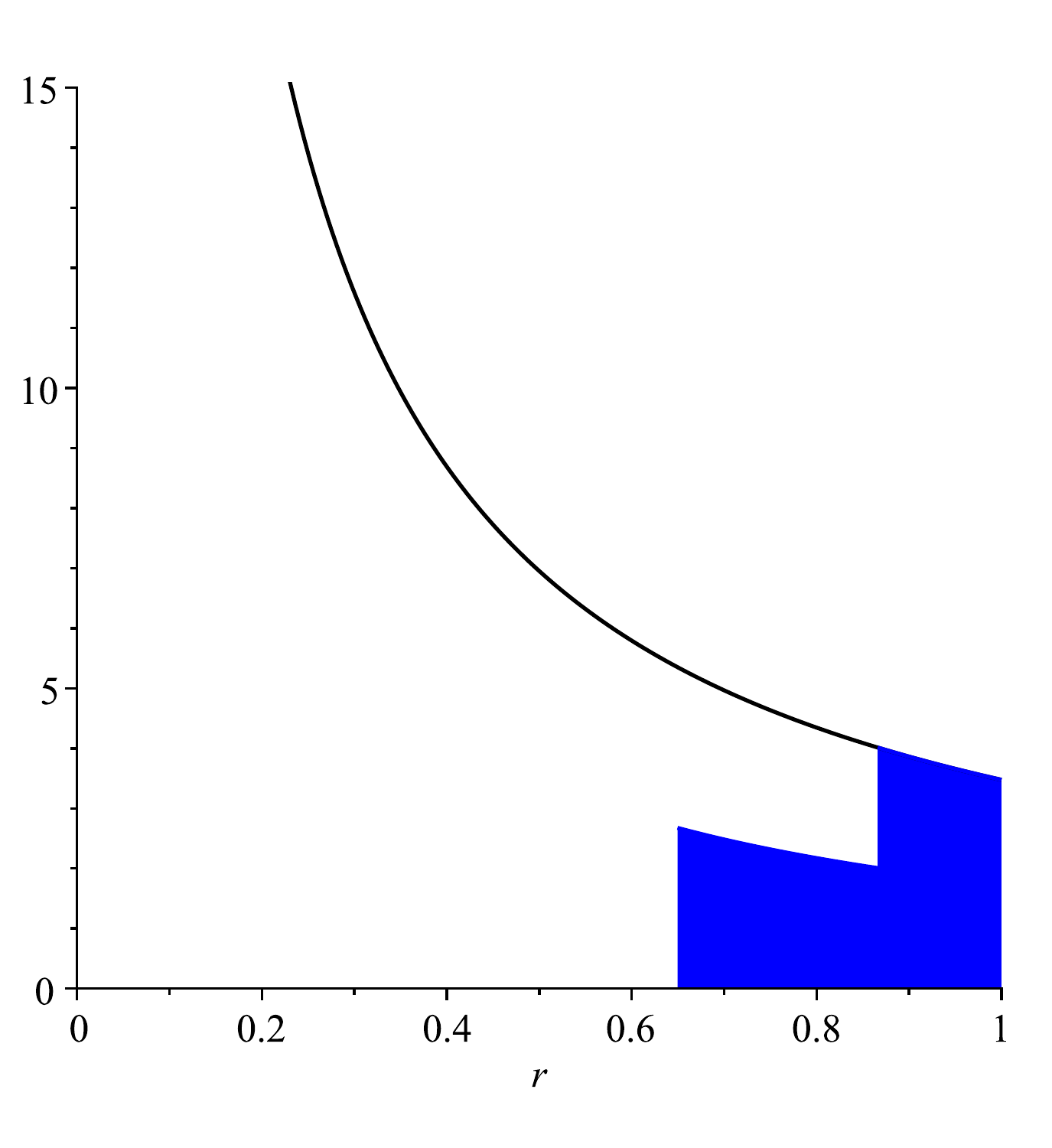}
\caption{The density of the radial part of the constrained equilibrium measure  for $q=3/4$, $\alpha=-2$ (left), and $\alpha=-3/4$ (right).
The black curve is the constraint.}
\label{fig:5}
\end{figure}
  
\subsection{Little $q$-Jacobi polynomials}
Little $q$-Jacobi polynomials are given by \cite[\S 14.12]{KLS}
\[   p_n(x;a,b|q) = {}_2\phi_1 \left( \left. \begin{array}{c} q^{-n}, abq^{n+1} \\ aq \end{array} \right| q, qx \right), \]
where $0 < aq < 1$ and  $bq <1$. Observe that the little $q$-Laguerre polynomials can be obtained from the little $q$-Jacobi polynomials by taking $b=0$
and the $q$-Bessel polynomials can be obtained from the little $q$-Jacobi polynomials by
\[    \lim_{a \to 0} p_n(x;a,-b/aq|q) = y_n(x;b;q).  \]
The orthogonality for the little $q$-Jacobi polynomials is
\[    \sum_{k=0}^\infty \frac{(bq;q)_k}{(q;q)_k} (aq)^k p_m(q^k;a,b|q) p_n(q^k;a,b|q) =
   \frac{(abq^2;q)_\infty (1-abq)(aq)^n(q,bq;q)_n}{(aq;q)_\infty (1-abq^{2n+1}) (ab,abq;q)_n} \delta_{m,n}, \]
so that
\[    w_k = \frac{(bq;q)_k}{(q;q)_k} a^k = w(q^k;a,b), \qquad  w(x;a,b) = \frac{(qx;q)_\infty}{(bqx;q)_\infty} x^\alpha, \]
where $q^\alpha = a$. For this weight we have
\[   \lim_{n \to \infty} \frac{1}{n^2} \log w(x^n;a,b) = 0, \qquad x \in (0,1), \]
so that Theorem \ref{thm:2.1} holds if we keep $\alpha$ and $\beta$ fixed.
In particular one has
\[  \lim_{n \to \infty} \gamma_n^{1/n^2} = q^{-1/2}. \]
This can easily be verified by using the explicit expression
\[  \gamma_n(a,b) = q^{-\frac{n(n-1)}{2}} \frac{(abq;q)_{2n}}{\sqrt{(abq;q)_n (aq;q)_n}} 
     \sqrt{ \frac{(aq;q)_\infty (1-abq^{2n+1})}{(abq^2;q)_\infty (1-abq)(aq)^n (q,q)_n (bq;q)_n}} . \]
The asymptotic distribution of the points $x_{1,n}^{1/n},\ldots,x_{n,n}^{1/n}$ is then given by the measure
with density (see Fig. \ref{fig:2}) 
\[   (\nu^\sigma)'(t) = \frac{-1}{\log q} \frac{1}{t}, \qquad t \in [q,1].  \]
This means that these zeros are dense on $[q,1]$ and the constraint $\sigma$ holds on the full support of 
this equilibrium measure.
 
We can let the parameters $a$ and $b$ depend on $n$. Recall that $0 < aq < 1$ and $bq <1$, so that putting $a = q^{2n\alpha}$ and $b=q^{2n\beta}$ only allows $\alpha \geq 0$ and $\beta \geq 0$. In that
case nothing new happens since $\frac{1}{n^2} \log (bqx^n;q)_\infty \to 0$ for $0< x \leq 1$, so that the parameter $b$ does not appear in the external field $Q(x) = -\alpha \log x$. So, as was the case for the little $q$-Laguerre polynomials, one has
\[  \lim_{n \to \infty} \gamma_n^{1/n^2} = q^{-1/2 - \alpha}, \]
and the asymptotic distribution of the points $x_{1,n}^{1/n},\ldots,x_{n,n}^{1/n}$ is given by the measure $\nu^\sigma$.
A more interesting situation is to take $b = -q^{2n\beta}$. Again $\beta \geq 0$ does not influence the external field $Q$, but for $\beta < 0$ 
the $\beta$ will be present in the external field. For this we use the following lemma.

\begin{lemma}  \label{lem:6.2}
Let $0 < q < 1$, then
\begin{equation}  \label{limb}
  \lim_{n \to \infty} \frac{1}{n^2} \log (-q^{nc};q)_\infty = 
                       \begin{cases} 0, & \textrm{if $c \geq 0$}, \\
                    \ds -\frac{c^2}2 \log q, & \textrm{if $c < 0$}.
                       \end{cases}
\end{equation}             
\end{lemma}

\begin{proof}
For $c >0$ we have
\[      \lim_{n \to \infty} (-q^{nc};q)_\infty = (0;q)_\infty = 1, \]
and for $c=0$ we have that $(-1;q)_\infty$ does not depend on $n$. This settles the lemma for $c \geq 0$.
For $c < 0$ we take $m=\lfloor -nc \rfloor$ and write $(-q^{nc};q)_\infty = (-q^{nc};q)_m (-q^{nc+m};q)_\infty$. Since $-nc \leq m < -nc+1$ we see that
$nc+m \geq 0$, so that the factor $(-q^{nc+m};q)_infty$ will not influence the limit in \eqref{limb}. Hence we need to work out
\[  \lim_{n \to \infty} \frac{1}{n^2} \log (-q^{nc};q)_m.  \]
Note that
\begin{eqnarray*}
     (-q^{nc};q)_m &=& (1+q^{nc})(1+q^{nc+1})(1+q^{nc+2}) \cdots (1+q^{nc+m-1}) \\
                   &=&  q^{mnc +1+2+\cdots+m-1} (1+q^{-nc})(1+q^{-nc-1})(1+q^{-nc-2}) \cdots (1+q^{-nc-m+1}) \\
                   &=&  q^{mnc + m(m-1)/2} (-q^{-nc-m+1})_m, 
\end{eqnarray*}
and since $-nc-m+1 >0$ we see that
\[    \lim_{n \to \infty} \frac{1}{n^2} \log (-q^{nc};q)_m = \lim_{n \to \infty} \frac{mnc+m(m-1)/2}{n^2} \log q  \]
and since $m/n \to -c$ as $n \to \infty$, the limit in \eqref{limb} follows easily.
\end{proof}

Using this lemma, we see that for $a=q^{2n\alpha}$ $(\alpha \geq 0)$ and $b = -q^{2n\beta}$
\[  \lim_{n \to \infty} - \frac{1}{n^2} \log w(x^n;q^{2n\alpha},-q^{2n\beta})
        = -2\alpha \log x + \begin{cases} 0, & \textrm{if $x \leq q^{-2\beta}$}, \\
                                   \ds  -\frac12 \left( \frac{\log x}{\log q} + 2 \beta \right)^2 \log q, & \textrm{if $x > q^{-2\beta}$}.
                      \end{cases} \]
so that the external field is
\[    Q(x) =  -\alpha \log x + \begin{cases} 0, & \textrm{if $x \leq q^{-2\beta}$}, \\
                                         -\frac14 \left( \frac{\log x}{\log q} + 2 \beta \right)^2 \log q, & \textrm{if $x > q^{-2\beta}$}.
                               \end{cases}.  \]
For this external field one has
\[   xQ'(x) = \begin{cases}   - \alpha, & \textrm{if $x \leq q^{-2\beta}$}, \\
                              -\alpha-\beta - \frac12 \frac{\log x}{\log q}, & \textrm{if $q^{-2\beta}<x\leq 1$}.
              \end{cases}   \]
Observe that $xQ'(x) > 0$ whenever $x > q^{-2\alpha+\beta}$ and $xQ'(x)=1$ when $x=q^{-2(\alpha+\beta+1)}$, provided $q^{-2\alpha+\beta} \leq 1$ and
$q^{-2(\alpha+\beta+1)} \leq 1$. So in a similar way as for the $q$-Bessel polynomials, we find that the constrained equilibrium measure
has a radial part with density
\[   v(r) = \begin{cases}  \ds \frac{-1}{\log q} \frac{1}{r} \chi_{[q,1]}, & \textup{ if } \alpha+\beta \geq - \frac12, \\[10pt]
                           \ds \frac{-1}{2\log q} \frac1r \chi_{[q^{-2(\alpha+\beta)},1]}  - \frac{1}{2\log q} \frac1r \chi_{[q^{2(\alpha+\beta+1)},1]}, 
                                & \textup{ if } -1 < \alpha+\beta < -\frac12, \\[10pt]
                           \ds \frac{-1}{2\log q} \frac{1}{r} \chi_{[q^{-2(\alpha+\beta)},q^{-2(\alpha+\beta+1)}]}, & \textup{ if } \alpha+\beta \leq -1. 
                \end{cases}  \]
These densities are the same as those in Fig.~\ref{fig:2} and Fig.~\ref{fig:5} except that $\alpha$ needs to be replaced by $\alpha+\beta$.
Roughly speaking, the parameter $q=q^{2n\alpha}$ (with $\alpha >0$) pushes the zeros to the right, whereas the parameter $b=-q^{2n\beta}$ (with $\beta < 0$) pushes the zeros to the left. Sometimes (when $\alpha+\beta \geq -1/2$) the pushing is not observed since the constraint doesn't give enough room to push.

\begin{verbatim}
Walter Van Assche
Quinten Van Baelen
Department of Mathematics
KU Leuven
Celestijnenlaan 200B box 2400
BE-3001 Leuven
walter.vanassche@kuleuven.be
\end{verbatim}

\end{document}